\documentclass[11pt]{amsart}
\usepackage{amsmath,amsthm, amscd, amssymb, amsfonts, mathtools}
\usepackage[all]{xy}
\usepackage{mathrsfs}
\usepackage[inline]{enumitem}
\usepackage{multicol}
\usepackage[T2A,T1]{fontenc}
\usepackage{rotating}
\usepackage{hyperref}

\newcommand{\polring}{R}
\newcommand{\nucleo}{\mathbf R}
\newcommand{\supbos}{\, \natural \,}
\newcommand{\braibos}{\sharp}
\newcommand{\hopfuno}{\mathtt{H}}
\newcommand{\hopfdos}{\mathtt{K}}
\newcommand{\hopfdoble}{\mathtt{D}}
\newcommand{\hopfq}{/ \hspace{-3pt} /}

\newcommand{\spl}{\mathfrak{sl}}

\newcommand{\verma}{M}

\allowdisplaybreaks

\usepackage[ansinew]{inputenc}
\usepackage{graphicx,fancyhdr}

\usepackage[dvips, dvipsnames, usenames]{color}

\numberwithin{equation}{section}
\newtheorem{theorem}{Theorem}[section]
\newtheorem{lemma}[theorem]{Lemma}

\newtheorem{prop}[theorem]{Proposition}

\theoremstyle{definition}

\newtheorem{example}[theorem]{Example}

\theoremstyle{remark}
\newtheorem{remark}[theorem]{Remark}

\newcommand{\pf}{\begin{proof}}
\newcommand{\epf}{\end{proof}}

\newcommand{\ku}{ \Bbbk}
\newcommand{\fp}{\mathbb F_p}

\newcommand{\kut}{ \ku^{\times}}

\newcommand{\G}{\mathbb G}

\newcommand{\I}{\mathbb I}

\newcommand{\N}{\mathbb N}

\newcommand{\Z}{\mathbb Z}

\newcommand{\D}{\mathcal{D}}

\newcommand{\cA}{\mathcal{A}}

\newcommand{\cB}{\mathcal{B}}
\newcommand{\cO}{\mathcal{O}}
\newcommand{\cW}{\mathcal{W}}

\newcommand{\cD}{\mathcal{D}}

\newcommand{\cH}{\mathcal{H}}
\newcommand{\cK}{\mathcal{K}}

\newcommand{\cL}{\mathcal{L}}
\newcommand{\Pc}{{\mathcal P}}
\newcommand{\cR}{\mathcal{R}}
\newcommand{\Ss}{{\mathcal S}}

\newcommand{\cV}{\mathcal{V}}
\newcommand{\cX}{\mathcal{X}}
\newcommand{\cZ}{\mathcal{Z}}

\newcommand{\g}{\mathfrak g}

\newcommand{\ugo}{\mathfrak u}

\newcommand{\ad}{\operatorname{ad}}

\newcommand{\Aut}{\operatorname{Aut}}
\newcommand{\AuH}{\Aut_{\text{Hopf}}}
\newcommand{\AubH}[1]{{\Aut_{\text{Hopf}}}_{#1}^{#1}}

\newcommand{\car}{\operatorname{char}}

\newcommand{\id}{\operatorname{id}}
\newcommand{\gr}{\operatorname{gr}}
\newcommand{\GK}{\operatorname{GKdim}}
\newcommand{\Hom}{\operatorname{Hom}}

\newcommand{\Svec}{\operatorname{\textsf{sVec}}}

\newcommand{\Irr}{\operatorname{Irrep}}
\newcommand{\Ind}{\operatorname{Ind}}

\newcommand{\gpuno}{\mathtt{C}}
\newcommand{\gpdos}{\mathtt{G}} 
\newcommand{\gptres}{\mathtt{N}} 
\newcommand{\morgpo}{\mathtt{F}}

\newcommand{\toba}{\mathscr{B}}
\newcommand{\scL}{\mathscr{L}}
\newcommand{\scK}{\mathscr{K}}

\newcommand{\ot}{\otimes}

\newcommand{\sydA}{{}^{\mathcal{A} }_{\mathcal{A} }\mathcal{YDS}}

\newcommand{\customyd}[1]{{}^{ #1 }_{ #1 }\mathcal{YD}}
\newcommand{\yd}[1]{{}^{ #1 }_{ #1 }\mathcal{YD}}
\newcommand{\customyds}[1]{{}^{ #1 }_{ #1 }\mathcal{YDS}}
\newcommand{\genpi}{\epsilon}
\newcommand{\gen}{\mathbf{g}}
\newcommand{\gentilde}{\widetilde{\mathbf{g}}}
\newcommand{\gentwop}{\gamma}

\newcommand{\Gb}{\mathbf G}
\newcommand{\Hb}{\mathbf H}
\newcommand{\Bb}{\mathbf B}

\newcommand{\xb}{x}
\newcommand{\ub}{u}

\DeclareRobustCommand{\stirling}{\genfrac []{0pt}{}}

\newcommand{\xrightarrowdbl}[2][]{%
\xrightarrow[#1]{#2}\mathrel{\mkern-14mu}\rightarrow
}

\newcounter{tabla}\stepcounter{tabla}

\begin{document}

\title[On the double of the super Jordan plane]
{On the double of the (restricted) super Jordan plane}

\author[Nicol\'as Andruskiewitsch and H\'ector Pe\~na Pollastri]
{Nicol\'as Andruskiewitsch and H\'ector Pe\~na Pollastri}

\thanks{This material is based upon work supported by the National Science Foundation under
Grant No. DMS-1440140 while N. A. was in residence at the Mathematical Sciences
Research Institute in Berkeley, California, in the Spring 2020 semester. 
The work of N. A. and H. P. P. was partially supported by CONICET and Secyt (UNC)}

\address{ Facultad de Matem\'atica, Astronom\'ia y F\'isica,
Universidad Nacional de C\'ordoba. CIEM -- CONICET. 
Medina Allende s/n (5000) Ciudad Universitaria, C\'ordoba, Argentina}
\email{andrus|hpollastri@famaf.unc.edu.ar}

\begin{abstract}
We consider the super Jordan plane, a braided Hopf algebra introduced--to the best of our knowledge--in 
\cite{aah-triang}, and its restricted version in odd characteristic introduced in \cite{aah-oddchar}. 
We show that 
their Drinfeld doubles give rise naturally to Hopf superalgebras
justifying a posteriori the adjective \emph{super} given in \emph{loc. cit}.
These Hopf superalgebras are extensions of super commutative ones by the enveloping, respectively
restricted enveloping, algebra of $\mathfrak{osp}(1|2)$.
\end{abstract} 

\maketitle

\setcounter{tocdepth}{1}
\tableofcontents

\section*{Introduction}

\subsubsection*{The context}
Let $\ku$ be an algebraically closed field. The super Jordan plane is the 
graded algebra $\cB$ 
presented by generators $x_1, x_2$ with defining relations
\begin{align}\label{eq:def-super-jordan}
&x_1^2, & &x_2 x_{21} - x_{21} x_2 - x_1 x_{21},
\end{align}
where $x_{21} = x_2 x_1 + x_1 x_2$. It is known that $\cB$ has Gelfand-Kirillov dimension 2.
The braided vector space $(V,c)$ with basis $\{x_1, x_2\}$ and braiding
\begin{align}\label{eq:braidingsuperjordan}
c(x_i\ot x_1) &= - x_1\ot x_i, & c(x_i\ot x_2) &= (- x_2 +x_1)\ot x_i,& i &=1,2,
\end{align}
determines a structure of braided Hopf algebra on $\cB$ in the sense of \cite{Tak}.

The Jordan plane and the super Jordan plane play a central role 
in the study of pointed Hopf algebras over abelian groups with finite
Gelfand-Kirillov dimension in  \cite{aah-triang}, assuming $\car \ku = 0$.
 
\medbreak
Let $p > 2$ be a prime.
Assume  that $\car \ku = p$.
Motivated by \cite{clw} that deals with the restricted Jordan plane,
it is shown in \cite{aah-oddchar} that many analogues of the Nichols algebras from \cite{aah-triang} have  
finite dimension (see \cite{ABDF} for examples in characteristic 2).
Among them the \emph{restricted} super Jordan plane i.e. the algebra
presented by $x_1, x_2$ with relations \eqref{eq:def-super-jordan} and
\begin{align}\label{eq:def-super-jordan-restricted}
& x_{21}^p, && x_2^{2p}.
\end{align}
We began the study of the Drinfeld doubles of (suitable bosonizations of) 
the Jordan plane and its restricted analogue in \cite{ap}
showing among other results that they fit into exact 
sequences relating the enveloping algebra of $\spl_2(\ku)$, the algebras of functions on some algebraic groups
and their restricted analogues. 
In the present paper we carry out a similar analysis of the super Jordan plane and its restricted version.
We start in Section \ref{sec:recollections} with  brief discussions on various topics needed later:
Yetter-Drinfeld supermodules, restricted Lie superalgebras and Nichols algebras.
In Section \ref{sec:preliminaries} we first record presentations of duals and doubles of finite-dimensional Hopf 
algebras arising as bosonizations; this is essentially straightforward but useful for further computations.
Then we present the objects of our interest: the super Jordan plane, its restricted version and their duals.
Finally we deal with different descriptions as bosonization of the same Hopf algebra. This allows to define 
alternatively the Hopf superalgebras $\widetilde{\cD}$ and $\cD$ discussed below.

\subsubsection*{The double of the super Jordan plane}
Here we just need that $\car \ku \neq 2$. As in \cite{aah-triang}
we realize the braided vector space $(V,c)$ with $c$ given by \eqref{eq:braidingsuperjordan} 
as Yetter-Drinfeld module over the group algebra $\ku \Z$, hence 
we have the Hopf algebra $\hopfuno \coloneqq \cB\#\ku\Z$. 
Then we consider the dual super Jordan plane $\cB^d$ described just before Lemma \ref{lema:dual-jordan}.
The Sweedler dual of $\ku \Z$ is spanned by the characters of $\Z$ and the Lie algebra of the one-dimensional torus.
Hence its smallest Hopf subalgebra realizing $\cB^d$ is isomorphic to $\ku[\zeta] \otimes \ku C_2$. 
Here $C_N$ stands for the cyclic group of order $N$ and $\zeta$ is an indeterminate.
Thus we may consider $\hopfdos \coloneqq \cB^d \# (\ku[\zeta] \otimes \ku C_2)$ and
define $\hopfdoble = \hopfuno \bowtie \hopfdos^{\operatorname{op}}$ with respect to a suitable pairing between $\hopfuno$ and $\hopfdos^{\operatorname{op}}$. It turns out that
there exists a Hopf superalgebra $\widetilde{\cD}$ such that $\hopfdoble \simeq  \widetilde{\cD} \# \ku C_2$;
this justifies a posteriori the adjective super given to $\cB$. Thus the study of 
$\hopfdoble$ reduces to that of $\widetilde{\cD}$.
We present basic properties of $\widetilde{\cD}$ in Proposition \ref{prop:Db} including the defining relations
and a PBW-basis. Next we show in Theorem \ref{prop:Db-exact} that $\widetilde{\cD}$ fits into an exact  sequence of
Hopf superalgebras
$\cO(\mathfrak{G})\hookrightarrow \widetilde{\cD}\twoheadrightarrow U(\mathfrak{osp}(1|2))$,
where $\mathfrak{G}$ is an algebraic supergroup explicitly described.
For the next result, Theorem \ref{thm:Z}, we need $\car \ku = p > 2$; then $\widetilde{\cD}$ is free module of finite rank over a central Hopf subalgebra $Z = \cO(\Bb)$ where $\Bb$ is a solvable connected algebraic group.
We close Section \ref{sec:double-superjordan} establishing some basic ring-theoretical properties of 
$\widetilde{\cD}$.

\subsubsection*{The double of the restricted super Jordan plane} In Section \ref{sec:double-restricted-superjordan}
we assume that $\car \ku = p > 2$. We realize $(V,c)$ with the braiding \eqref{eq:braidingsuperjordan} in 
$\yd{\ku C_{2p}}$. 
Let $D(H)$ be the Drinfeld double of the bosonization $H = \toba(V)\#\ku C_p$.
Again there exists a Hopf superalgebra $\cD$ 
such that $D(H) \simeq  \cD \# \ku C_2$, thus we may focus on $\cD$.
We present basic properties of $\cD$ in Proposition \ref{prop:cD-as-Hopf-superalgebra} 
and show in Theorem \ref{prop:cD-as-super-abelian-extension} an exact  sequence of
Hopf superalgebras
$\nucleo \xhookrightarrow[]{\iota} \cD  \xrightarrowdbl[]{\pi} \ugo(\mathfrak{osp}(1|2))$
where $\nucleo$ is a local commutative Hopf algebra and $\ugo(\mathfrak{osp}(1|2))$
is the restricted enveloping algebra. 
We conclude that the simple $\cD$-modules are the same as those of $\ugo(\mathfrak{osp}(1|2))$
and we present them as quotients of Verma modules.
See Theorem \ref{th:irrep-restricted-super-jordan} and Proposition \ref{prop:irrep-cD}.

\medbreak
The extensions of Hopf superalgebras mentioned above fit into a 9-term commutative diagram
where all columns and rows are exact sequences:
\begin{align}\label{eq:diagram-exact sequences}
\begin{aligned}
\xymatrix{ & \cO(\Gb)  \ar@{^{(}->}[r] \ar@{^{(}->}[d] & \cO(\Bb) \ar@{^{(}->}[d] \ar@{->>}[r] & \cO(\Gb_a^3) \ar@{^{(}->}[d]
\\
& \cO(\mathfrak{G})  \ar@{^{(}->}[r] \ar@{->>}[d] & \widetilde{\cD} \ar@{->>}[r] \ar@{->>}[d] & U(\mathfrak{osp}(1|2)) \ar@{->>}[d]
\\
& \nucleo \ar@{^{(}->}[r] & \cD \ar@{->>}[r] & \ugo(\mathfrak{osp}(1|2))
}
\end{aligned}
\end{align}
See Theorem \ref{thm:diagram-exact sequences}. This is an analogue of \cite[(0.2)]{ap} for the Jordan plane.

\section{Recollections}\label{sec:recollections}

\subsection{Conventions}
We denote by $\fp$ the field of $p$ elements. If $\ell < n \in\N_0$, then we set $\I_{\ell, n}=\{\ell, \ell +1,\dots,n\}$, $\I_n = \I_{1, n}$. Let $A$ be an algebra and $a_1,\dots,a_n\in A$, $n\in\N$. We denote by $\ku\langle a_1,\dots,a_n\rangle$ the subalgebra generated by $a_1,\dots,a_n$. 
We identify $\cV^*\ot \cV^*\simeq (\cV\ot \cV)^*$, where $\cV$ is a finite-dimensional vector space, by
\begin{align}\label{eq:dual-tensor-product}
\langle f\ot f', x\ot y \rangle &= \langle f,x\rangle \langle f',y \rangle, & f,f'\in \cV^*, x,y\in \cV.
\end{align}

The cyclic group of order $n$ is denoted by $C_n$ and the infinite cyclic group by $\Gamma$. They are always written multiplicatively. The center of a group $G$ is denoted $Z(G)$; similarly for the center of an algebra.

Let $L$ be a Hopf algebra. The kernel of the counit $\varepsilon$ is denoted $L^+$, the antipode (always assumed bijective) by $\Ss$ or by $\Ss_L$, the
space of primitive elements of $L$ is denoted by $\Pc(L)$ and the group of group-likes by $G(L)$.
The space of $g,h$-primitives is $\Pc_{g,h}(L) =\{x \in L: \Delta(x) = x\otimes h + g \otimes x\}$ where $g,h \in G(L)$.
We assume that the reader has some familiarity with the theory of Hopf algebras particularly with the 
notions of Yetter-Drinfeld module and bosonization (or biproduct), see e.g. \cite{Rad-libro}.
The category of Yetter-Drinfeld modules over $L$ is denoted by $\customyd{L}$.

The category of super vector spaces is denoted by $\Svec$. If $\cX\in\Svec$, and $x\in\cX_i$, 
then we write $|x| = i$. We set $|\cX| \coloneqq \cX_0 \cup \cX_1$. 
The symmetric tensor category $\Svec$ is identified with a full tensor subcategory of $\customyd{\ku C_2}$. 
An object $M\in\customyd{\ku C_2}$ is in $\Svec$ if the two following conditions are satisfied:
\begin{enumerate}[leftmargin=*,start=0]
\item\label{svect-equivalence-1} For every $a\in M$ such that $\genpi \xrightharpoonup[]{} a = a$, $\delta(a) = 1 \ot a$ (then $a$ is even).

\item\label{svect-equivalence-2} For every $a\in M$ such that $\genpi \xrightharpoonup[]{} a = -a$, $\delta(a) = \genpi \ot a$ (then $a$ is odd).
\end{enumerate}

Let $\cA$ be a superalgebra, i.e. an algebra in $\Svec$; $\cA$ is super commutative if $a b = (-1)^{|a||b|} b a$ for all $a,b\in|\cA|$. 

The algebra of regular functions on an (affine) algebraic super group $G$ is denoted $\cO(G)$. This is by definition a finitely generated super commutative Hopf algebra. See \cite{Masuoka} for the definitions.
As usual, $\Gb_a$ is the additive algebraic group $(\ku,+)$ and $\Gb_m$ is the multiplicative algebraic group $(\kut,\cdot)$.

Let $\polring$ be the polynomial algebra $\ku[\zeta]$ with the unique Hopf algebra structure such that $\zeta\in\Pc(\polring)$. Let $\polring_p$ be the quotient Hopf algebra $\polring/(\zeta^p- \zeta)$.

Let $V = \ku\{X_1,\dots,X_m\}$ be a vector space of dimension $m\in\N$. We denote by $\Lambda(V) = \Lambda(X_1,\dots,X_m)$ the exterior algebra of $V$. 

We denote by $[k]^{[n]} \coloneqq \prod_{i=1}^{n} (k+i-1)$ $k\in\ku$, $n\in\N$ the raising factorial, we also define $[k]^{[0]} \coloneqq 1$ for every $k\in\ku$. The unsigned Stirling numbers $\stirling{n}{k}$ are defined as the coefficients of the `raising factorial' polynomial: 
\begin{align*}
[X]^{[n]} = \prod_{i=1}^{n} (X+i-1) = \sum_{k=0}^n \stirling{n}{k} X^k \in \Z[X].
\end{align*}

Recall that a sequence of morphisms of Hopf  (super)algebras
$A\xhookrightarrow[]{\iota} C \xrightarrowdbl[]{\pi} B$
is exact, and $C$ is an extension of $A$ by $B$, if 
\begin{multicols}{2}
	\begin{enumerate}[leftmargin=*,label=\rm{(\roman*)}]
\item\label{suc-exacta-1} $\iota$ is injective.
\item\label{suc-exacta-2} $\pi$ is surjective.
\item\label{suc-exacta-3} $\ker\pi = C\iota(A)^+$.
\item\label{suc-exacta-4} $\iota(A) = C^{\operatorname{co} \pi}$.
	\end{enumerate}
\end{multicols}
We write $B = C \hopfq A$ i.e. $B$ is the Hopf algebra quotient of $C$ by $A$.

\begin{remark}\label{remark-exact-sequence-hopf}
	If $A\xhookrightarrow[]{\iota} C$ is faithfully flat and $\iota(A)$ is stable by the left adjoint action of $C$ then
	\ref{suc-exacta-1}, \ref{suc-exacta-2} and \ref{suc-exacta-3} imply \ref{suc-exacta-4}, see \cite[1.2.5, 1.2.14]{ad}, \cite{Sch}.
\end{remark}

\subsection{A brief review of Yetter-Drinfeld supermodules}\label{subsub:brief-YDsuper}
In this Subsection, $\car \ku \neq 2$. See \cite{aay-pointedsuphopf} for more details.
Let $\cA$ be a Hopf superalgebra. The category of Yetter-Drinfeld supermodules over $\cA$,
denoted by $\customyds{\cA}$, consists of super vector spaces $\cX$ such that:
\begin{enumerate}[leftmargin=*,label=\rm{(\roman*)}]
\item $\cX$ is a left supermodule over $\cA$ with action $\rightharpoonup$.
\item $\cX$ is a left supercomodule over $\cA$ with coaction $\delta$.
\item For every $a\in|\cA|$ and $u\in |\cX|$:
\begin{align*}
\delta(a\rightharpoonup x) = (-1)^{|v_{(-1)}|(|a_{(2)}|+ |a_{(3)}|) + |a_{(2)}||a_{(3)}|} 
a_{(1)} x_{(-1)} \Ss(a_{(3)})\ot a_{(2)}\rightharpoonup x_{(0)}.
\end{align*}
\end{enumerate}

Then $\customyds{\cA}$ is a braided tensor category. Namely, if 
$\cX,\cZ \in \customyds{\cA}$, then the super vector space $\cX\ot \cZ$ 
is an object in $\customyds{\cA}$  via
\begin{align*}
&a\rightharpoonup (x\ot z) = (-1)^{|a_{(2)}||x|} a_{(1)}\rightharpoonup x \ot a_{(2)} \rightharpoonup z, &\\
&\delta(x \ot z) = (-1)^{|x_{(0)}||z_{(-1)}|} x_{(-1)} z_{(-1)}\ot(x_{(0)}\ot z_{(0)}), & \forall x\in|\cX|, z\in|\cZ|.  
\end{align*}
The braiding $c\colon\cX\ot \cZ\longrightarrow \cZ\ot\cX$ is given by:
\begin{align}\label{eq:braiding-yds}
c(x\ot z) &= (-1)^{|x_{(0)}||z|} x_{(-1)}\rightharpoonup z \ot x_{(0)}.
\end{align}


\medbreak
There is an embedding functor of braided tensor categories
\begin{align}\label{eq:embedding-functor}
\mathfrak{i}\colon \sydA &\longrightarrow \customyd{\cA\#\ku C_2}, &
\cX \longmapsto \mathfrak{i}(\cX)
\end{align} 
where $\mathfrak{i}(\cX) = \cX$ as vector space with
action $\rightharpoonup$ and coaction $\delta$ as:
\begin{align*}
&\begin{aligned}
a\# \genpi^k \rightharpoonup x &= (-1)^{|x|k} a\rightharpoonup x,\\
\delta(x) &= x_{(-1)}\# \genpi^{|x_{(0)}|}\ot x_{(0)},
\end{aligned} & \forall a\in\cA, x\in|\cX|, k\in\I_{0, 1}.
\end{align*}

\subsubsection*{The super bosonization \cite{aay-pointedsuphopf}}
Let $\cR$ be a Hopf algebra in  $\sydA$. 
The comultiplication is written $\Delta_\cR(h)= h^{(1)}\ot h^{(2)}$ for $h\in\cR$. 
The  Hopf superalgebra $\cR\supbos  \cA$  is $\cR \ot \cA$ as a super vector space with structure given by:
\begin{align*}
(h \supbos  a)(f\supbos  b) &\coloneqq (-1)^{|a_{(2)}||f|} h(a_{(1)}\rightharpoonup f)\supbos  a_{(2)} b, \\
\Delta(h\supbos  a) &\coloneqq (-1)^{|(h^{(2)})_{(0)}||a_{(1)}|} h^{(1)}\supbos  (h^{(2)})_{(-1)} a_{(1)} \ot (h^{(2)})_{(0)} \supbos  a_{(2)},\\
\varepsilon(h\supbos  a) &= \varepsilon_\cR(h) \varepsilon_\cA(a),\quad \quad\quad 1\coloneqq 1_{\cR}\supbos  1_{\cA},\\
\Ss(h\supbos  a) &= (-1)^{|h_{(0)}||a|} (1\supbos  \Ss_\cA(h_{(-1)}a))(\Ss_\cR(h_{(0)})\supbos  1).
\end{align*} 
There is a canonical isomorphism $(\cR\supbos  \cA)\# \ku C_2 \simeq \mathfrak{i}(\cR)\#(\cA \# \ku C_2)$.

\subsection{Restricted  Lie superalgebras} A Lie superalgebra $\g = \g_0\oplus \g_1$ is restricted if 
$\g_0$ is a restricted Lie algebra  with $p$-operation $x\mapsto x^{[p]}$ such that 
\begin{align*}
[x^{[p]},z] &= (\ad x)^p (z) & \text{for all } x&\in \g_0, z\in\g.
\end{align*}
We refer to \cite[Chapter 3]{bmpz} for a detailed exposition.
The \emph{restricted enveloping algebra} of the restricted Lie superalgebra $\g$ is defined as
\begin{align*}
\ugo(\g)\coloneqq U(\g)/ \langle x^{[p]} - x^p : x\in\g_0 \rangle,
\end{align*}
where $U(\g)$ is the enveloping algebra. Assume for simplicity that $\dim \g < \infty$
and let $\{x_1,\dots,x_r\}$, $\{y_1,\dots,y_k\}$ be bases of $\g_0$ and $\g_1$ respectively. 
Then the Hopf superalgebra $\ugo(\g)$ has dimension $p^r2^k$; indeed it has a PBW-basis
\begin{align*}
\{x_1^{n_1}\cdots x_r^{n_r} y_1^{m_1}\cdots y_k^{m_k}\colon n_1,\dots,n_r\in\I_{0,p-1}, m_1,\dots,m_k\in\I_{0, 1}\}.
\end{align*}

\subsubsection{The ortho-symplectic Lie superalgebra $\g = \mathfrak{osp}(1|2)$ }
For our purposes we recall its structure:
$\g_0 \simeq \mathfrak{sl}_2(\ku)$ with Cartan generators $\{e, f, h\}$; $\g_1 = \ku\{\psi_+, \psi_-\}$ is the natural $\mathfrak{sl}_2(\ku)$-module, hence the bracket is
\begin{align*}
[e, \psi_+] &= 0, & [h, \psi_+] &= \psi_+, & [f, \psi_+] &= \psi_-,\\
[e, \psi_-] &= \psi_+, & [h, \psi_-] &= -\psi_-, & [f, \psi_-] &= 0,
\\
[\psi_+,\psi_+] &= 2e, & [\psi_-,\psi_-] &= -2f, & [\psi_+,\psi_-] &= -h. 
\end{align*}

\emph{Assume now that $\car \ku = p$.}
The algebra $\g$ has a $p$-structure given by
\begin{align*}
e^{[p]} &= 0, & h^{[p]} &= h, & f^{[p]} &= 0.  
\end{align*}
Below we consider the enveloping algebra $U(\mathfrak{osp}(1|2))$ and its restricted version $\ugo(\mathfrak{osp}(1|2))$
which are Hopf superalgebras with suitable PBW-bases.

\subsection{Nichols algebras} A braided vector space is a pair $(\cV, c)$ where $\cV$ is a vector space and
$c \in GL(\cV^{\otimes 2})$ satisfies the braid equation 
\begin{align*}
(c\ot \id)(\id\ot c)(c\ot \id) = (\id\ot c)(c\ot \id)(\id\ot c).
\end{align*}  
Then the braid group $\mathbb B_n$ acts on $\cV^{\otimes n}$; the Nichols algebra $\toba(\cV)$ is defined taking the quotient of
the quantum symmetrizer in each degree. See \cite{A-leyva} for details.
A realization of $(\cV, c)$ over a Hopf algebra $L$ is a structure of Yetter-Drinfeld $L$-module on $\cV$
such that $c$ coincides with the categorical braiding. Then $\toba(\cV)$ is a Hopf algebra in $\customyd{L}$ and the bosonization $\toba(\cV) \# L$ is a Hopf algebra.
If $\cV$ is finite-dimensional, then the dual vector space $\cV^*$ is braided with the transpose braiding $c^*$--recall the identification \eqref{eq:dual-tensor-product}:
\begin{align}\label{eq:braiding-transpose}
\langle c^*(f\ot f'), x\ot y \rangle &= \langle f\ot f', c(x\ot y)\rangle, & f,f'\in \cV^*, x,y\in \cV.
\end{align}

\medbreak
Similarly a realization of $(\cV,c)$ over a Hopf superalgebra $\cA$ is a structure of Yetter-Drinfeld 
$\cA$-supermodule on $\cV$ such that $c$ coincides with the categorical braiding. 
Then $\toba(\cV)$ is a Hopf algebra in $\customyds{\cA}$ and the bosonization $\toba(\cV) \# \cA$ is a Hopf 
superalgebra.
Notice that there is an isomorphism $\mathfrak{i}(\toba(\cV))\simeq \toba(\mathfrak{i}(\cV))$  where $\mathfrak{i}$
is as in \eqref{eq:embedding-functor}, see \cite[\S 1.7]{aay-pointedsuphopf}.

\medbreak 
If $(\cV,c)$ admits a realization over a Hopf superalgebra, then necessarily $\cV$ is $C_2$-graded
and $c$ preserves the grading of $\cV \otimes \cV$; such $(\cV,c)$ might be called a super braided vector space,
a concept already present in \cite[(53), p. 1610]{Kulish-Sklyanin}. Indeed there is bijective correspondence between 
super braided vector spaces and solutions of the super Yang-Baxter equation, see \emph{loc. cit.}

\section{Preliminaries}\label{sec:preliminaries}

\subsection{The dual and the double of a bosonization}
Let $L$ be a Hopf algebra and $\cV\in\customyd{L}$; we assume that $\dim L <\infty$ and $\dim\toba(\cV)<\infty$.
In this subsection we describe explicitly the dual and the double of the bosonization
$A \coloneqq \toba(\cV)\# L$. In particular we show that
$A^* \simeq \toba(\cV^*)\# L^*$,
and that $\cV^*$ has the transpose braiding \eqref{eq:braiding-transpose}.
We need some notation. First, we have morphisms of Hopf algebras 
$A\overset{\pi }{\underset{\iota}{\rightleftarrows}} L$ such that $\pi \iota = \id$.
Next we fix
\begin{itemize}[leftmargin=*]\renewcommand{\labelitemi}{$\circ$}
\item  a basis  $\{v_1,\dots,v_n\}$ of $\cV$; its dual basis is denoted by $\{w_1,\dots,w_n\}$;
\item a basis $\{h_1,\dots,h_m\}$ of $L$; its dual basis is denoted by $\{f_1,\dots,f_m\}$.
\end{itemize}
The braided tensor categories $\customyd{L}\simeq \customyd{L^*}$ are equivalent via the functor 
\begin{align*}
F\colon \customyd{L}\longrightarrow\customyd{L^*}
\end{align*} defined as follows. If $\cX\in \customyd{L}$, then $F(\cX) = \cX$ with structure
\begin{align*}
f\rightharpoonup v &= \langle f, \Ss(v_{(-1)}) \rangle v_{(0)}, & \delta(v) &= \sum_{i=0}^m \Ss^{-1}(f_i)\ot h_i\rightharpoonup v, & v\in \cX, f\in L^*,
\end{align*}
see \cite[Prop. 2.2.1]{g2}.
Let $\cW\coloneqq F(\cV^*)$.
\begin{lemma} \cite[Lemma 2.6]{g}
As a  braided vector space, $\cW\in\customyd{L^*}$ is isomorphic to $(\cV^*,c^*)$. Hence
$\dim\toba(\cW) = \dim\toba(\cV) <\infty$. \qed
\end{lemma}

\begin{prop}\label{prop:dual-of-a-bosonization}
$A^* \simeq \toba(\cW)\# L^*$.
\end{prop}
\begin{proof}
By transposing the maps $\pi$ and $\iota$ above,  
we get maps $A^*\overset{\iota^*}{\underset{\pi ^*}{\rightleftarrows}} L^*$ with $\iota^*\pi ^* = \id$. Then
$A^*\simeq \cR \# L^*$, where $\cR = (A^*)^{\operatorname{co}\iota^*}$ is a braided Hopf algebra in $\customyd{L^*}$. Since $A = \bigoplus_{n\in\N_0} A^n$ is graded, so are
$A^*= \bigoplus_{n\in\N_0} (A^*)^n$ and $\cR= \bigoplus_{n\in\N_0} \cR^n$. We proceed in three steps:

\subsubsection*{(i) Find a basis of $\cR^1\in\customyd{L^*}$.}
Since $A^1 \simeq \cV\ot L$, $(A^*)^1 \simeq\cV^*\ot L^*$ as a vector space. 
Given $w\in \cV^*$ and $f\in L^*$ we set $w\underline{\#}f \coloneqq w\ot f \in(A^*)^1$ by
\begin{align*}
\langle w\underline{\#} f, v\# h\rangle &= \langle w, v\rangle \langle f, h\rangle, & \forall v\in \cV, h\in L.
\end{align*}
Then $\{w_i\underline{\#} f_j\colon i\in \I_n, j\in\I_m \}$ is a basis of $(A^*)^1$. Let $\overline{w}_i \coloneqq w_i\underline{\#} \varepsilon$, $i\in \I_n$. We claim that $\{\overline{w}_1,\dots,\overline{w}_n\}$ is a basis of $\cR^1$. 
For this, we compute  
\begin{align*}
\langle\Delta(\overline{w}_i), 1\# h_\ell \ot v_j\# h_k\rangle &= \langle w_i,h_\ell\rightharpoonup v_j\rangle \varepsilon(h_k),\\
\langle \Delta(\overline{w}_i), v_j\# h_k\ot 1\# h_\ell\rangle &= \varepsilon(h_k h_\ell) \delta_{i,j},
\end{align*}
for all $i,j \in \I_n$, $k,\ell\in \I_m$. Since $\Delta$ preserves the grading, we have
\begin{align*}
\Delta(\overline{w}_i) &= \sum_{\ell,k = 1}^{m} \varepsilon(h_k h_\ell) w_i\underline{\#} f_k \ot f_\ell + \sum_{\ell,k = 1}^m \sum_{j=1}^n \langle w_i,h_\ell\rightharpoonup v_j\rangle \varepsilon(h_k)  f_\ell \ot w_j\underline{\#}f_k\\
&=\overline{w}_i \ot 1 +  \sum_{\ell = 1}^m \sum_{j=1}^n  \langle w_i,h_\ell\rightharpoonup v_j\rangle f_\ell \ot \overline{w}_j.
\end{align*}
Then $\overline{w}_i\in\cR^1$ because $\iota^*(\overline{w}_j)=0$ for all $j\in \I_n$. 
Since the $\overline{w}_i$'s are linearly independent and $\dim \cR^1 = \dim \toba(\cV)^1 = \dim \cV$, the claim follows.

\subsubsection*{(ii) $\cR^1\simeq \cW$ in $\customyd{L^*}$.}
To compute the action and coaction of $\cR^1$ we need the multiplication of elements of the form $(w\underline{\#}f) f'$ and $f'(w\underline{\#}f)$ for $f,f'\in L^*$, $w\in \cV^*$. Since $A^0\simeq L$, $(A^*)^0\simeq L^*$, we see that
\begin{align*}
\langle (w\underline{\#} f) f', v\# h\rangle &= \langle w\underline{\#} ff', v\# h\rangle, \\
\langle f' (w\underline{\#} f), v\# h\rangle &= \langle f'_{(1)}\rightharpoonup w \underline{\#} f'_{(2)}f, v\#h\rangle
\end{align*}
for all $v\in \cV$, $h\in L$. Then 
\begin{align*}
f'(w\underline{\#} f) &= f'_{(1)}\rightharpoonup w \underline{\#} f'_{(2)}f 
&&\text{and}&  (w\underline{\#} f) f' &= w\underline{\#} ff'
\end{align*}
for $w\in \cV^*, f,f'\in L^*$. Hence the action of $L^*$ in $\cR^1$ is given by
\begin{align*}
f\rightharpoonup \overline{w}_i &= f_{(1)} (w_i\underline{\#} \varepsilon) \Ss(f_{(2)}) = f_{(1)}\rightharpoonup w_i \underline{\#} f_{(2)}\Ss(f_{(3)}) \\
&=f_{(1)}\rightharpoonup w_i\underline{\#} \varepsilon(f_{(2)})\varepsilon = f\rightharpoonup w_i \underline{\#} \varepsilon \\
&= \sum_{j=1}^{n} \langle f, (v_j)_{(-1)}\rangle \langle w_i, (v_j)_{(0)}\rangle \overline{w}_j,\quad i\in\I_n, f\in L^*.
\end{align*}
The formula for the coaction follows from the comultiplication above: 
\begin{align*}
\delta(\overline{w}_i) &= (\iota^*\ot \id)\circ \Delta(\overline{w}_i) = \sum_{\ell=1}^{m}\sum_{j=1}^{n} \langle w_i, h_\ell\rightharpoonup v_j\rangle f_\ell\ot \overline{w}_j, & i\in \I_n.
\end{align*}
Thus $\cW\simeq \cR^1$ in $\customyd{L^*}$.

\subsubsection*{(iii) $\cR \simeq \toba(\cW)\simeq\toba(\cV^*)$} Let $\cR'$ be the braided Hopf subalgebra of $\cR$ generated by $\cR^1$. Then $\cR'$ is a pre-Nichols algebra of $\cR^1 \simeq \cW$. Thus
\begin{align*}
\dim \toba(\cW) \leq \dim \cR' \leq \dim \cR = \dim \toba(\cV) = \dim\toba(\cW);
\end{align*}
hence $\cR' = \cR = \toba(\cW)$ and the result follows.
\end{proof}

Next we describe the relations of the Drinfeld double $D(A)$ of $A$. 
Recall that this is the Hopf algebra 
whose underlying coalgebra is $ A\ot A^{* \operatorname{op}}$ and with multiplication defined as follows.
Let $a\bowtie r:= a \otimes r$ in $D(A)$. Then 
\begin{align*}
(a\bowtie r)(a'\bowtie r') &= \big\langle r_{(1)},a'_{(1)}\big\rangle \big\langle r_{(3)}, \Ss(a'_{(3)}) \big\rangle (aa'_{(2)}\bowtie r'r_{(2)}), & 
\end{align*}
$a, A'\in A$, $r, r'\in A^{* \operatorname{op}},$
where $r r' = m(r\ot r')$ is  in $A^*$, not in $A^{* \operatorname{op}}$. 

\begin{remark}\label{rem:double} \cite{dt}
If $A$ is not finite-dimensional, we may define its double
with respect to another Hopf algebra $B$ provided with
a skew-pairing,  i.~e. a linear map 
$\tau\colon A\otimes B\to \ku$ satisfying
\begin{align}\label{eq:skew-pairing}
\begin{aligned}
\tau(a\widetilde a\ot b) &= \tau(a \ot b_{(1)}) \tau(\widetilde a \ot b_{(2)}),&
\tau(1 \ot b) &= \varepsilon(b), & a,\widetilde a &\in A,
\\
\tau(a \ot \widetilde bb) &= \tau(a_{(1)} \ot b) \tau(a_{(2)} \ot \widetilde b),&
\tau(a \ot 1) &= \varepsilon(a), &b, \widetilde b &\in B.
\end{aligned}
\end{align}
Let $\sigma\colon (A \ot B)\ot (A \ot B)\to  \ku$ be the
2-cocycle  associated to $\tau$, where $A \ot B$ has the structure of tensor product Hopf algebra, that is
\begin{align*}
\sigma(a\ot b, \widetilde{a}\ot \widetilde{b}) &= \varepsilon(a)\varepsilon(\widetilde{b}) \tau(\widetilde{a}\ot b), & a,\widetilde{a} &\in A, b,\widetilde{b} &\in B.
\end{align*}
Then the double of $A$ (with respect to $B$ and $\tau$) is the
Hopf algebra $A\ot B$ twisted by $\sigma$, i.~e. $A \bowtie B \eqqcolon (A\ot B)_{\sigma}$.

\end{remark}

\medbreak
We fix the following presentations for $\toba(\cV)$, $\toba(\cW)^{\operatorname{op}}$ and $D(L)$:
\begin{align}\label{eq:general-presentation-by-generators-and-relations-for-Nichols}
\begin{aligned}
\toba(\cV) &= \ku \langle v_1,\dots,v_n \vert r_1,\dots,r_{n_1}\rangle, 
\\ \toba(\cW)^{\operatorname{op}} &= \ku \langle w_1,\dots,w_n \vert r'_1,\dots,r'_{n_2}\rangle, 
\\
D(L) &=\ku \langle s_1,\dots,s_{m_0}, s'_1,\dots, s'_{m_1} \vert \widetilde{r}_1,\dots,\widetilde{r}_{n_3}\rangle,
\end{aligned}
\end{align}  
with $s_1,\dots,s_{m_0}$ and $s'_1,\dots, s'_{m_1}$ generators of $L$ and $(L^*)^{\operatorname{op}}$ respectively. Then we have the following presentation for $D(A)$.

\begin{prop}\label{prop:the-double-in-general}
The algebra $D(A)$ is presented by generators $v_1,\dots,v_n$, $w_1,\dots,w_n$, $s_1,\dots,s_{m_0}$, $s'_0,\dots, s'_{m_1}$ with relations \eqref{eq:general-presentation-by-generators-and-relations-for-Nichols} and

\begin{align}\label{eq:relations-doble-general}
\begin{aligned}
s'_i v_j &= \langle (s'_i)_{(1)}, (v_j)_{(-1)} \rangle (v_j)_{(0)} (s'_i)_{(2)}, \quad (i,j)\in\I_{m_1}\times \I_n, \\
w_i s_j &= \sum_{\ell=1}^n \langle w_i, (s_j)_{(1)}\rightharpoonup v_\ell\rangle (s_j)_{(2)} w_\ell, \quad (i,j)\in\I_{n}\times \I_{m_0}, \\
w_i v_j &= \langle w_i, v_j\rangle 1 + \sum_{\ell=1}^n \langle w_i, (v_j)_{(-1)}\rightharpoonup v_\ell\rangle (v_j)_{(0)} w_\ell\\
+\sum_{k=1}^m \sum_{\ell,t=1}^n &\langle w_i, (v_j)_{(-2)}\rightharpoonup v_\ell\rangle \langle w_\ell, h_k\rightharpoonup v_t \rangle
\langle \overline{w}_t, \Ss((v_j)_{(0)}) \rangle (v_j)_{(-1)} f_k,\\
& \quad (i,j)\in\I_n\times \I_n. 
\end{aligned}
\end{align}
Hence $D(A)$ admits a triangular decomposition 
\begin{align}\label{eq:triag-decomp-in-general}
D(A) \simeq \toba(\cV)\ot D(L) \ot \toba(\cW)^{\operatorname{op}}.
\end{align}
\end{prop}
\begin{proof}
Let $A$ be the algebra presented by generators and relations as above. By definition of $D(A)$ and the formulas above there exists an algebra map $A\twoheadrightarrow  D(A)$. Hence $\dim A \geq D(A)$. By definition of $A$ there exist morphisms
\begin{align*}
\psi_1&\colon \toba(\cV)\longrightarrow A,& \psi_2&\colon \toba(\cW)^{\operatorname{op}} \longrightarrow A, &\psi_3&\colon D(L)\longrightarrow A.
\end{align*} 
Let $m\colon A\ot A\longrightarrow A$ be the multiplication. We have a linear map
\begin{align}\label{eq:triag-decomp}
\phi\colon\toba(\cV)\ot D(L) \ot \toba(\cW)^{\operatorname{op}} \xrightarrow{\psi_1\ot\psi_2\ot\psi_3} A\ot A\ot A \xrightarrow{m\circ(m\ot\id)} A
\end{align}
which is surjective by \eqref{eq:relations-doble-general}, since every product of the generators can be rewritten to get 
a sum of elements of the form $abc$ with $a\in\toba(\cV)$, $b\in D(L)$ and $c\in \toba(\cW)^{\operatorname{op}}$. Thus 
\begin{align*}
\dim D(A) = \dim \toba(\cV)\cdot \dim D(L) \cdot\dim\toba(\cW)^{\operatorname{op}} \geq \dim A,
\end{align*}
hence $D(A)\simeq A$. As for \eqref{eq:triag-decomp-in-general}, 
\eqref{eq:triag-decomp} provides the desired isomorphism.
\end{proof}

\subsection{Objects of interest} 
In this Subsection we introduce the braided vector spaces of our interest and their Nichols algebras, cf. \cite{aah-triang,aah-oddchar}.

\subsubsection*{Braided vector spaces}
The braided vector space  $\left(\cV(-1, 2), c\right) \eqqcolon \left(V, c_{V}\right)$, called the $-1$-block in \cite{aah-triang}, 
has a basis $\{\xb_1,\xb_2\}$ and braiding
\begin{align}\tag{\ref{eq:braidingsuperjordan}}
c(x_i\ot x_1) &= - x_1\ot x_i, & c(x_i\ot x_2) &= (- x_2 +x_1)\ot x_i,& i &=1,2.
\end{align}
For simplicity we denote
$\left(W, c_{W}\right) \coloneqq \left(\cV(-1,2)^*, c^*\right)$,  cf. \eqref{eq:braiding-transpose}.
Let $\{u_1,u_2\}$ be the basis of $W$ given by $\langle u_i, x_j \rangle = 1 - \delta_{i,j}$, $i,j\in \I_2$.
In this basis $c_{W}$ is
\begin{align}\label{eq:braidingsuperjordandual}
c_{W}( u_1\ot u_i) &= -u_i\ot  u_1, & 
c_{W}( u_2\ot u_i) &= u_i \ot ( u_1 - u_2), & i &=1,2.
\end{align}
Then $(W,c_W)\simeq (V,c_V^{-1})$ as braided vector spaces, via  $u_1\mapsto x_1$, $u_2\mapsto -x_2$.

\subsubsection*{The super Jordan plane} This is the graded algebra
\begin{align*}
\cB = \ku \langle x_1, x_2 \vert x_1^2, x_2 x_{21} - x_{21} x_2 - x_1 x_{21}\rangle
\end{align*}  
with the  braided Hopf algebra structure extending the braiding \eqref{eq:braidingsuperjordan}, cf. \cite{Tak}.

\begin{lemma} 
	\begin{enumerate}[leftmargin=*,label=\rm{(\alph*)}]
\item \cite{aah-triang} ($\car \ku \neq 2$) The ordered monomials
\begin{align}\label{eq:monomials:jordan}
x_1^{n_1} x_{21}^{n_{21}} x_2^{n_2}
\end{align}
with $(n_1, n_{21}, n_2) \in\I_{0, 1} \times \N_{0} \times \N_{0}$ form a basis of $\cB$ and $\GK \cB = 2$. 

\medbreak
\item ($\car \ku = 0$) \cite{aah-triang} $\cB$ is isomorphic to  $\toba(V)$ as braided Hopf algebras. 

\medbreak
\item ($\car \ku = p$) \cite{aah-oddchar}
The restricted super Jordan plane $\cB/ \langle x_{21}^p, x_2^{2p}\rangle$
is isomorphic to $\toba(V)$ as braided Hopf algebras. The ordered monomials \eqref{eq:monomials:jordan} with 
$(n_1, n_{21}, n_2) \in\I_{0, 1} \times \I_{0, p-1} \times \I_{0,2p-1}$ form a basis of $\toba(V)$.  \qed
	\end{enumerate}
\end{lemma}

Set $u_{21} \coloneqq u_1 u_2 + u_2 u_1$.
We define the \emph{dual super Jordan plane}  as 
the algebra $\cB^d$ presented by generators $u_1$ and $u_2$ with defining relations
\begin{align}\label{eq:relations-H-dual}
u_1^2 &=0, &  u_2 u_{21} &= u_{21} u_2 - u_1 u_{21}
\end{align}  
with the  braided Hopf algebra structure extending the braiding \eqref{eq:braidingsuperjordandual}, cf. \cite{Tak}.
Observe that $\cB^d$ is not isomorphic to $\cB$ as braided Hopf algebras.
The \emph{restricted dual Jordan plane} is the quotient of $\cB^d$ by the relations
\begin{align}\label{eq:relations-H-dual2}
u_2^{2p}&=0, & u_{21}^p&=0.
\end{align}  

\begin{lemma} \label{lema:dual-jordan}
	\begin{enumerate}[leftmargin=*,label=\rm{(\alph*)}]
\item\label{item:Jordandual} ($\car \ku = 0$) $\cB^d\simeq \toba(W)$ as braided Hopf algebras. 

\medbreak
\item\label{item:Jordandual-restricted} ($\car \ku = p$) The restricted dual Jordan plane 
is isomorphic to $\toba(W)$ as braided Hopf algebras. It has dimension $4p^2$; indeed the 
monomials 
\begin{align}\label{eq:monomials:jordan-dual}
u_1^{n_1} u_{21}^{n_{21}} u_2^{n_2}
\end{align}
with $(n_1, n_{21}, n_2) \in\I_{0, 1} \times \I_{0, p-1} \times \I_{0,2p-1}$ form a basis of $\toba(W)$.  
	\end{enumerate}
\end{lemma}

\pf Since $\cB^d \simeq \cB$ as algebras, the monomials \eqref{eq:monomials:jordan-dual}
with  $(n_1, n_{21}, n_2) \in\I_{0, 1} \times \N_{0} \times \N_{0}$ 
form a linear basis of $\cB^d$. 
It is easy to see that 
$u_1^2$ and $u_{21} u_2  - u_2 u_{21}  + u_1 u_{21}$ are primitive in $T(W)$, hence 
$\cB^d$ is a pre-Nichols algebra of $W$. 
Now $\toba(V)$ and $\toba(W)$ have the same graded dimension by \cite[Lemma 2.6]{g}. Then \ref{item:Jordandual} follows. 
For \ref{item:Jordandual-restricted} we prove by induction that 
\begin{align}\label{eq:coproduct-dual-super-jordan}
\begin{aligned}
\Delta(u_{21}^n) =& \sum_{\ell=0}^{1}\sum_{k=0}^{n-\ell}\binom{n-\ell}{k} n^\ell u_1^\ell 
u_{21}^{n-k-\ell} \ot u_1^\ell u_{21}^k, \\
\Delta(u_2^{2n}) =& \sum_{\ell=0}^1\sum_{k=0}^{n-\ell} \sum_{t=0}^{k} \binom{n-\ell}{k}\binom{k}{t} n^\ell
[k-n+\ell]^{[t]}
u_2^{2(n-k)-\ell}\ot u_1^\ell u_{21}^t u_2^{2(k-t)},
\end{aligned}
\end{align}
for all $n\in\N_0$. 
So $u_{21}^p, u_2^{2p}\in \Pc(\cB^d)$ and $\cB^d/(u_{21}^p, u_2^{2p})$ is a pre-Nichols algebra of $W$; 
hence  $\cB^d/(u_{21}^p, u_2^{2p}) \simeq \toba(W)$ by dimension counting. 
\epf

\subsection{Change of bosonization}\label{sub:change-boso}

Let $L$ be a Hopf algebra and let $\gpuno$ be a group provided with a group homomorphism 
$\gpuno \to \AuH (L)$. Then $L$ is a Hopf algebra in $\yd{\ku \gpuno}$ (with trivial coaction)
and we may consider the smash product $L \rtimes \ku \gpuno \coloneqq L \# \ku \gpuno$.
Let $\xymatrix{L \rtimes \ku \gpuno \ar@<0.5ex>@{->}[r]^{\quad\pi} &  \ku \gpuno \ar@<0.5ex>@{->}[l]^{\quad \iota}}$ be the natural projection and inclusion. We show that under certain conditions, $L \rtimes \ku \gpuno$ can be 
alternatively described as $\scL^{\morgpo} \# \ku \gpuno$ for a genuine braided Hopf algebra 
$\scL^{\morgpo}$ in $\yd{\ku \gpuno}$. 

Namely, suppose that 
$L = R \# U$ where $U$ is a Hopf algebra and $R$ is a Hopf algebra in $\yd{U}$. 
Let $\AubH{U} (R)$ be the group of automorphisms of
Hopf algebras in $\yd{U}$, i.e. preserving multiplication, comultiplication, action and coaction.
There is a morphism of groups 
\begin{align*}
\AubH{U} (R) &\to \AuH(L),& 
\AubH{U} (R) \ni \varsigma &\mapsto \varsigma \otimes \id \in \AuH(L). 
\end{align*}
We fix a group $\gpuno$ and a morphism of groups
$\gpuno \to \AuH (L)$  factorizing through $\AubH{U} (R)$.
Furthermore we assume that $U = U'\otimes \ku \gpdos$ where $U'$ a Hopf algebra and $\gpdos$ is a group.

\begin{prop} \label{prop:change-boso}
\begin{enumerate}[leftmargin=*,label=\rm{(\alph*)}]
\item\label{item:change-boso-a}  Given $\morgpo \in \Hom_{\emph{gps}} \left(\gpdos, Z(\gpuno)\right)$, there exists 
a Hopf algebra $\scL^{\morgpo}$ in $\yd{\ku \gpuno}$ such that  $L \rtimes \ku \gpuno \simeq \scL^{\morgpo} \# \ku \gpuno$.

\medbreak
\item\label{item:change-boso-b} Let  
$\varpi^{\morgpo}: \gpdos \times \gpuno \to \gpuno$ be the morphism of groups given by
$\varpi^{\morgpo} (\gamma, c) = \morgpo(\gamma) c$, $\gamma \in \gpdos$, $c \in \gpuno$,
and let
\begin{align*}
\gpdos^{\morgpo} \coloneqq \ker \varpi^{\morgpo} = \{(\gamma, \morgpo(\gamma)^{-1}): \gamma \in \gpdos\}. 
\end{align*}
\end{enumerate}
\noindent 
Then $\scL^{\morgpo}$ decomposes as $\scL^{\morgpo} \simeq R^{\morgpo} \braibos( U' \otimes \ku \gpdos^{\morgpo})$ (braided bosonization in $\yd{\ku \gpuno}$) where 
$R^{\morgpo}  = R$ as a subalgebra and $U'$-module but 
with $\gpdos^{\morgpo} \simeq \gpdos$ acting by  
\begin{align*}
\gamma \rightharpoonup r  &= \morgpo(\gamma)^{-1} \cdot \left(\gamma \cdot r\right),& 
r \in R, \gamma &\in \gpdos.
\end{align*}

\begin{enumerate}[leftmargin=*, resume,label=\rm{(\alph*)}]
\item\label{item:change-boso-c} Assume that 
\begin{enumerate}[leftmargin=*,label=\rm{(\roman*)}]
\item\label{item:change-boso-c-1} $\gpuno$ and $\gpdos$ are abelian and $\morgpo$ has a section 
$\vartheta \in \Hom_{\emph{gps}} \left(\gpuno, \gpdos\right)$,

\item\label{item:change-boso-c-2} For every $r\in R$ and $c\in \gpuno$, $c \cdot r = \vartheta(c) \cdot r$.
\end{enumerate}
\end{enumerate}
Let $\gptres^{\morgpo} \coloneqq \ker \morgpo \times \{e\}$.  Then the subalgebra 
$R^{\morgpo} \braibos U' \otimes \ku \gptres^{\morgpo}$ has a structure of Hopf algebra denoted $\cL^{\morgpo}$
such that $\ku \gpuno \hookrightarrow \scL^{\morgpo} \twoheadrightarrow \cL^{\morgpo}$ is exact.
Furthermore $R^{\morgpo}$ is a Hopf algebra in $\yd{U' \otimes \ku \gptres^{\morgpo}}$.
\end{prop}

\pf \ref{item:change-boso-a}: 
Let $\pi^{\morgpo}: L \rtimes \ku \gpuno \to \ku \gpuno$ be  the Hopf algebra projection given by
\begin{align*}
\pi^{\morgpo} (r \# u \ot \gamma \otimes c) &= \varepsilon(r)\varepsilon(u) \, \morgpo(\gamma) c,& 
r \in R,\ u \in U', \  &\gamma \in \gpdos, \ c\in  \gpuno. 
\end{align*}
Then $\pi^{\morgpo}\iota = \id_{\ku \gpuno}$ and the claim follows from \cite{Rad-libro} taking the subalgebra
\begin{align*}
\scL^{\morgpo} \coloneqq \left(L \rtimes \ku \gpuno\right)^{\text{co }\pi^{\morgpo}} = R \cdot U' \ot \ku \gpdos^{\morgpo}.
\end{align*} 

\ref{item:change-boso-b} The Hopf algebra maps $\xymatrix{\scL^{\morgpo} \ar@<0.5ex>@{->}[r]^{p^{\morgpo}} &  U' \otimes \ku  \gpdos^{\morgpo} \ar@<0.5ex>@{->}[l]^{\imath^{\morgpo}}}$ given  by
\begin{align*}
&\begin{aligned}
p^{\morgpo} (r \# u \ot \gamma \morgpo(\gamma)^{-1}) &= \varepsilon(r)\varepsilon(u) (1\# \gamma \morgpo(\gamma)^{-1}),
\\
\imath^{\morgpo} (u \ot \gamma \morgpo(\gamma)^{-1}) &= 1 \# u \ot \gamma \morgpo(\gamma)^{-1},
\end{aligned}
& r \in R,\ u \in U', \  &\gamma \in \gpdos, 
\end{align*} 
satisfy $p^{\morgpo}\imath^{\morgpo} = \id_{\ku \gpdos^{\morgpo}}$. Hence 
$R^{\morgpo} \coloneqq \left(\scL^{\morgpo} \right)^{\text{co } p^{\morgpo}}$ and the claim follows.

\ref{item:change-boso-c} By \ref{item:change-boso-c-1},  
$\gpdos^{\morgpo} \simeq \gptres^{\morgpo} \times\gpuno^{\vartheta}$ where
$\gpuno^{\vartheta} = \{(\vartheta(c)^{-1}, c): c\in \gpuno \}$ and using also \ref{item:change-boso-c-1},
$\gpuno^{\vartheta}$ is central in $\scL^{\morgpo}$.
Now the multiplication provides a linear isomorphism
\begin{align*}
\scL^{\morgpo} \simeq R^{\morgpo} \braibos( U' \otimes \ku \gpdos^{\morgpo}) 
\simeq R^{\morgpo} \braibos\big( U' \otimes (\ku \gptres^{\morgpo} \times \ku \gpuno^{\vartheta})\big)
\simeq\left( R^{\morgpo} \braibos U' \otimes \ku \gptres^{\morgpo}\right) \otimes \ku \gpuno^{\vartheta}.
\end{align*}
Then the quotient $\cL^{\morgpo} \simeq \scL^{\morgpo}/\ku (\gpuno^{\vartheta})^+ \scL^{\morgpo}$ is isomorphic as an algebra
to the subalgebra $R^{\morgpo} \braibos U' \otimes \ku \gptres^{\morgpo}$. The rest is clear.
\epf

The situation that we have in mind is when $\gpuno = C_N$ and $R = \oplus_{i \in C_N} R_i$ 
has a $C_N$-grading of Hopf algebras in $\yd{U}$. For instance,
let $\cV \in \yd{\ku \gpdos}$ be $C_N$-graded and take $R = \toba(\cV)$ (or any pre-Nichols algebras whose defining relations are $C_N$-homogeneous).
We fix $\omega \in \G_N$ (what amounts in this case to fix a quasi-triangular structure 
on $\ku C_{N}$) and let act $C_N$ on $R_i$ by $\omega^i$.

We discuss two specific examples of interest in this paper.
In the rest of this Section $\car \ku \neq 2$, in particular $\car \ku = 0$ is allowed. 
Recall that $\polring = \ku[\zeta]$ with $\zeta$ primitive and $\Gamma \simeq \Z$.
We fix generators
\begin{align}\label{eq:generators-groups}
C_2 &= \langle \genpi \rangle, & \Gamma &= \langle \gen \rangle.
\end{align}

\begin{example}
We realize  $(V, c)$, with braiding \eqref{eq:braidingsuperjordan}, in $\customyd{\ku\Gamma}$ by
\begin{align}\label{eq:realization-yd-Gamma}
\gen\rightharpoonup \xb_1 &= -\xb_1, & \gen \rightharpoonup \xb_2 &= -\xb_2 + \xb_1, & \delta(\xb_i) &= \gen \ot \xb_i, 
& i &=1,2.
\end{align}
Since the ideal of $T(V)$ generated by the relations \eqref{eq:def-super-jordan} belongs to $\customyd{\ku\Gamma}$,
$\cB$ is a Hopf algebra in $\customyd{\ku\Gamma}$ hence we have 
\begin{align}\label{eq:def-hopfuno}
\hopfuno &\coloneqq \cB\#\ku\Gamma.
\end{align}
Now thinking on $V$ as purely odd and taking $\gpuno = C_2$, $\gpdos = \Gamma$, $L = \hopfuno$,
$U = \ku \Gamma$ and $\morgpo:\Gamma \to C_2$ the standard projection, we have
$\hopfuno \rtimes \ku C_2 \simeq \widetilde{\cH} \# \ku C_2$ where
\begin{align}\label{eq:def-cH}
\widetilde{\cH} \simeq \cB \supbos \ku\Gamma,
\end{align}
that corresponds to the realization of $V$ in $\customyds{\ku\Gamma}$ by
\begin{align}\label{eq:realization-yds-Gamma}
&\gen\rightharpoonup \xb_1 = \xb_1, \ \gen \rightharpoonup \xb_2 = \xb_2 - \xb_1, \ \delta(\xb_i) = \gen \ot \xb_i, \
|x_i| = 1,& i &\in \I_2.
\end{align}
Since $\morgpo$ admits no section, there is no further decomposition.
\end{example}

\begin{example}
Analogously we realize $(W, c_{W})$, cf. \eqref{eq:braidingsuperjordandual}, in $\customyd{\polring\ot\ku C_2}$ by
\begin{align} \label{eq:realization-yd-polring}
&\begin{aligned}
\zeta\rightharpoonup \ub_i &= \ub_i,& \genpi\rightharpoonup \ub_i, &= -\ub_i, \\
\delta(\ub_1) &= \genpi \ot \ub_1, & \delta(\ub_2) &= \genpi\ot \ub_2 - \zeta\genpi\ot \ub_1,
\end{aligned}
& i &=1,2.
\end{align}
As before $\cB^d$ is a Hopf algebra in $\customyd{\polring\ot\ku C_2}$ hence we have 
\begin{align}\label{eq:def-hopfdos}
\hopfdos &\coloneqq \cB^d \#(\polring\ot\ku C_2).
\end{align}
Now thinking on $W$ as a purely odd super vector space and taking $\gpuno = C_2\simeq \gpdos$, $L = \hopfdos$,
$U = \polring\ot\ku C_2$ and $\morgpo = \id_{C_2}$, we have
$\hopfdos\rtimes \ku C_2 \simeq \widetilde{\scK} \# \ku C_2$, where
$\widetilde{\scK} \simeq \cB^d \supbos ( \polring \otimes \ku C_2)$;
furthermore $\ku C_2$ is central and
\begin{align}\label{eq:def-cK}
\widetilde{\cK} 
&\coloneqq \widetilde{\scK} \hopfq \ku C_2  \simeq \cB^d \supbos \polring.
\end{align}
Here 
the last corresponds to the realization of $W$ in $\customyds{\polring}$ by
\begin{align}\label{eq:realization-yds-polring}
&\begin{aligned}
\zeta\rightharpoonup\ub_i &= \ub_i,& |u_i| &= 1,& i &=1,2,
\\
\delta(\ub_1) &= 1\ot \ub_1, & \delta(\ub_2) &= 1\ot \ub_2 - \zeta\ot \ub_1.&&
\end{aligned}\end{align}

\end{example}

The superalgebras $\widetilde{\cH}$ and $\widetilde{\cK}$ admit a suitable Hopf pairing that allows an alternative 
characterization of the Hopf superalgebra $\widetilde{\cD}$, see Remark \ref{rem:widetildeD-doble}.

\section{The double of the super Jordan plane}\label{sec:double-superjordan}

\subsection{The  definition} 
We define the  Drinfeld double $\hopfdoble$ of
$\hopfuno = \cB\#\ku\Gamma$, see \eqref{eq:def-hopfuno}, with respect to a suitable pairing with 
$\hopfdos = \cB^d\#(\polring \otimes \ku C_{2})$, see \eqref{eq:def-hopfdos}.
Then we show that there exists a Hopf superalgebra $\widetilde{\cD}$ such that $\hopfdoble \simeq \widetilde{\cD} \# \ku C_2$.
In this way  $\widetilde{\cD}$ is fundamental to the study of the Drinfeld double $\hopfdoble$; among other characteristics, it bears a triangular decomposition \eqref{eq:Dtilde-triang-decomp}. 

\medbreak
To start with the Hopf algebras $\hopfuno$ and $\hopfdos$ have PBW-basis, denoted by $\mathsf{B}_{\hopfuno}$, 
$\mathsf{B}_{\hopfdos}$, given by the ordered monomials \eqref{eq:monomials:jordan}
or \eqref{eq:monomials:jordan-dual} multiplied accordingly by elements of the groups $\Gamma$ and $C_2$,
or powers of $\zeta$.
The comultiplications of $\hopfuno$, respectively $\hopfdos$ satisfy $\xb_1, \xb_2 \in \Pc_{\gen, 1}(\hopfuno)$,
$\ub_1  \in \Pc_{\genpi, 1}(\hopfdos)$ and
\begin{align}\label{eq:comultiplication-Hb1} 
\Delta(\ub_2) &= \ub_2\ot 1 + \genpi\ot \ub_2 - \genpi\zeta\ot \ub_1.
\end{align}

\begin{lemma} 
The algebra $\hopfuno$ is presented by generators $\xb_1,\xb_2,\gen, \gen^{-1}$ and relations
\eqref{eq:def-super-jordan},
\begin{align}\label{eq:Relaciones-Hb}
\gen \xb_1 &= -\xb_1\gen , & \gen \xb_2 &= -\xb_2 \gen + \xb_1\gen, &
\gen ^{\pm 1} \gen ^{\mp 1} &= 1.
\end{align}
Also, $\hopfdos$ is presented by generators $ \ub_1, \ub_2,\zeta, \epsilon$ and relations
\eqref{eq:def-super-jordan} (in the $u_i$'s), 
\begin{align}
\label{eq:relation-epsilon}
\epsilon ^2 & = 1, & \epsilon \zeta & = \zeta\epsilon,
\\
\label{eq:Relaciones-Kb1}
\epsilon u_1 & = -u_1 \epsilon, & \epsilon u_2 & = -u_2 \epsilon,
\\
\label{eq:Relaciones-Kb}
\zeta \ub_1  &= \ub_1\zeta + \ub_1, &  \zeta \ub_2  &= \ub_2\zeta + \ub_2. \hspace{50pt}\qed
\end{align}
\end{lemma}

We define the Drinfeld double of $\hopfuno$ as  $\hopfdoble \coloneqq \hopfuno \bowtie \hopfdos^{\operatorname{op}}$, see Remark \ref{rem:double},
with respect to the unique skew-pairing $\tau: \hopfuno \otimes \hopfdos^{\operatorname{op}} \to \ku$
such that
\begin{align*}
\tau(\xb_1 \ot u_1) &= 0, & \tau(\xb_1\ot u_2) &= 1, & \tau(\xb_1\ot\zeta) &=0, & \tau(\xb_1\ot\genpi) &=0,
\\
\tau(\xb_2\ot \ub_1) &= 1, & \tau(\xb_2\ot \ub_2) &= 0, &  \tau(\xb_2\ot\zeta) &=0, & \tau(\xb_2\ot\genpi) &=0, 
\\
\tau(\gen ^{\pm 1}\ot \ub_1) &= 0, & \tau(\gen ^{\pm 1}\ot \ub_2) &=0,& \tau(\gen ^{\pm 1}\ot\zeta) &= \pm 1,
 & \tau(\gen ^{\pm 1}\ot\genpi) &= -1. 
\end{align*}

\begin{prop}\label{prop:presentation-tilde-D-0}
The algebra $\hopfdoble$ is presented by generators $\xb_1$, $\xb_2$, $\zeta$, $\gen$, $\genpi$, $\ub_1$, $\ub_2$ with relations \eqref{eq:def-super-jordan} (in the $ \xb_i$'s and in the $\ub_i$'s), \eqref{eq:Relaciones-Hb},
\eqref{eq:relation-epsilon}, \eqref{eq:Relaciones-Kb1},
\begin{align}
\label{eq:Relaciones-Kb-op}
\ub_1 \zeta  &= \zeta\ub_1 + \ub_1, & \ub_2\zeta &= \zeta\ub_2 + \ub_2,
\\
\label{eq:Relaciones-D0-part2}
\genpi \gen &= \gen \genpi,& \zeta \gen &= \gen \zeta,
\\
\label{eq:genpi-xb}
\genpi \xb_1 &= - \xb_1\genpi, &\genpi \xb_2 &= - \xb_2\genpi, 
\\ \label{eq:zeta-xi}
\zeta \xb_1 &= \xb_1 \zeta + \xb_1, & \zeta \xb_2 &= \xb_2 \zeta + \xb_2,
\\
\ub_2 \gen &= -\gen \ub_2 + \gen \ub_1, & \ub_1 \gen &= -\gen \ub_1,
\\
\label{eq:Relaciones-D0-part1}
\ub_1 \xb_1 &= -\xb_1 \ub_1, &
\ub_1 \xb_2 &= -\xb_2 \ub_1 +(1-\gen\genpi), 
\\  \ub_2 \xb_2 &= -\xb_2 \ub_2 +\gen\genpi \zeta + \xb_2 \ub_1, 
& \ub_2 \xb_1 &= -\xb_1 \ub_2 +(1-\gen\genpi) + \xb_1 \ub_1,
\end{align}

The family $\mathsf{B}_{\hopfdoble}$ consisting of the monomials
\begin{align}\label{eq:base_D_0}
\begin{aligned}
\xb_1^{n_1} \xb_{21}^{n_{21}} \xb_2^{n_2} \gen^n \zeta^m \ub_1^{m_1} \ub_{21}^{m_{21}} \ub_2^{m_2} \genpi^k 
\end{aligned}
\end{align}
with  $(k,n_1,m_1,n,n_2,n_{21},m,m_2,m_{21})\in \I_{0,1}^3\times \Z \times \N_0^5$ is a basis of $\hopfdoble$.
\end{prop}
\begin{proof}
Let $A$ be the algebra presented as in the statement;
then $A\twoheadrightarrow \hopfdoble$. As $\mathsf{B}_{\hopfdoble}$ is a basis of $ \hopfdoble$ by construction, the corresponding monomials are linearly independent in $A$. 
Let $S= \{\xb_1,\xb_{21},\xb_2,\gen,\zeta,\ub_1,\ub_{21},\ub_2, \genpi\}$ with $\xb_{21}$ and $\ub_{21}$  as before,
ordered by $\xb_1<\xb_{21}<\xb_2<\gen<\zeta<\ub_1<\ub_{21}<\ub_2<\genpi$. 
From the defining relations we have
\begin{align*}
&\ub_{21} \xb_1 = \xb_1 \ub_{21},  
\quad \genpi \ub_{21} =  \ub_{21}\genpi, \quad \ub_1 \xb_{21} =  \xb_{21} \ub_1,
\quad \genpi \xb_{21} = \xb_{21} \genpi,
\\
&\ub_{21} \xb_2 = \xb_2 \ub_{21} + (\gen\genpi + 1) \ub_1,  \quad 
\zeta x_{21} = \xb_{21} \zeta + 2 \xb_{21}, 
\\ &\ub_2 \xb_{21} = \xb_{21} \ub_2 - 2 \xb_{21} \ub_1 + \xb_1(\gen\genpi +1).
\end{align*}
If $a > b\in S$, then $ab$ is a linear combination of monomials $c_1\cdots c_s$ with $c_1\leq c_2\leq\cdots\leq c_s\in S$. Thus  the monomials \eqref{eq:base_D_0} 
generate $A$ and  $A\simeq \hopfdoble$. 
\end{proof}

Let $\gentilde\coloneqq \gen\genpi$ and $\widetilde{\cD} \coloneqq \ku \langle \xb_1, \xb_2, \ub_1, \ub_2, \gentilde, \zeta\rangle \hookrightarrow \hopfdoble$.
\begin{prop}\label{prop:Db}	
\begin{enumerate}[leftmargin=*,label=\rm{(\roman*)}] 
\item\label{prop:Db-1} A basis of $\widetilde{\cD}$ is given by the family $\mathsf{B}$ consisting of monomials 
\begin{align*}
&\xb_1^{n_1} \xb_{21}^{n_{21}} \xb_2^{n_2} \gentilde^n \zeta^m \ub_1^{m_1} \ub_{21}^{m_{21}} \ub_2^{m_2},
&(n_1,m_1,n,n_2,n_{21},m,m_2,m_{21})\in \I_{0,1}^2\times \Z \times \N_0^5.
\end{align*}

\medbreak
\item\label{prop:Db-4} $\widetilde{\cD}$ has a triangular decomposition i.e. a linear isomorphism induced by multiplication  
\begin{align}\label{eq:Dtilde-triang-decomp}
\widetilde{\cD} \simeq \cB \ot (\ku\Gamma\ot R) \ot (\cB^d)^{\operatorname{op}}.
\end{align}

\medbreak
\item\label{prop:Db-2} The algebra  $\widetilde{\cD}$ is presented by generators $\xb_1$, $\xb_2$, $\ub_1$, $\ub_2$, $\gentilde^{\pm1}$, $\zeta$ with defining relations \eqref{eq:def-super-jordan} (in the $x_i$'s and the $u_i$'s), \eqref{eq:Relaciones-Kb-op}, \eqref{eq:zeta-xi},
\eqref{eq:Relaciones-D0-part1} (with $\gentilde$ instead of $\gen \genpi$) and
\begin{align*}
\gentilde ^{\pm 1} \gentilde ^{\mp 1} &= 1, &\zeta \gentilde &= \gentilde \zeta, &
\gentilde \xb_1 &= \xb_1\gentilde, \\ 
\gentilde \xb_2 &= \xb_2 \gentilde - \xb_1\gentilde, &
\ub_2 \gentilde &= \gentilde \ub_2 - \gentilde \ub_1, & \ub_1 \gentilde &= \gentilde \ub_1.
\end{align*} 

\medbreak
\item\label{prop:Db-3}  $\widetilde{\cD}\#\ku C_2\simeq \hopfdoble$ as Hopf algebras;
here $\widetilde{\cD}$ is a Hopf superalgebra with co\-multiplication given by
$\gentilde\in G(\widetilde{\cD}), \, \zeta, \ub_1\in\Pc(\widetilde{\cD}), \,
\xb_1, \xb_2 \in \Pc_{\gentilde, 1}(\widetilde{\cD})$ and
\begin{align}
\label{eq:comultiplication-Dtilde}
\Delta_{\widetilde{\cD}}(\ub_2) &= \ub_2\ot 1 + 1\ot \ub_2 - \zeta\ot \ub_1,
\end{align}  

\end{enumerate}
\end{prop}
\begin{proof}
\ref{prop:Db-1} Argue as in the proof of Proposition \ref{prop:presentation-tilde-D-0}; in turn  
\ref{prop:Db-1} implies \ref{prop:Db-4}.

\ref{prop:Db-2} Let $A$ be the algebra as in the statement. Using the commutation relations we have a surjective map of algebras $\phi\colon A\twoheadrightarrow \widetilde{\cD}$. 
Clearly the family $\mathsf{B}'$ consisting of the  monomials analogous to those in $\mathsf{B}$ generates $A$ and is linearly independent because $\phi(\mathsf{B}') = \mathsf{B}$. Hence $A\simeq \widetilde{\cD}$. 

\ref{prop:Db-3} Let us identify $\ku C_2$ with the subalgebra of $\hopfdoble$ generated by $\genpi\in G(\hopfdoble)$
by $\iota\colon \ku C_2 \hookrightarrow \hopfdoble$.
There is a Hopf algebra map $\pi\colon \hopfdoble\twoheadrightarrow \ku C_2$  given by
\begin{align*}
\xb_1&\mapsto 0, & \xb_2&\mapsto 0, & \ub_1&\mapsto 0, & \ub_2&\mapsto 0, & \gen&\mapsto \genpi, & \zeta&\mapsto 0, & \genpi &\mapsto \genpi.
\end{align*}
Then $\pi \iota = \id$ and
$\hopfdoble \simeq \hopfdoble^{\operatorname{co} \pi} \# \ku C_2$. 
Clearly $\widetilde{\cD}\subseteq \hopfdoble^{\operatorname{co} \pi}$ because every generator of $\widetilde{\cD}$ is
coinvariant, and $\widetilde{\cD}$ is a braided Hopf algebra in $\customyd{\ku C_2}$ with comultiplication $\Delta_{\widetilde{\cD}}$ as in \eqref{eq:comultiplication-Dtilde}. But $\widetilde{\cD}\#\ku C_2$ contains  $\mathsf{B}_{\hopfdoble}$, thus $\widetilde{\cD} = \hopfdoble^{\operatorname{co} \pi}$. 
Since every generator of $\widetilde{\cD}$ is either even or odd, $\widetilde{\cD}$ is a Hopf superalgebra. 
\end{proof}

\begin{remark} \label{rem:widetildeD-doble}
Recall the Hopf superalgebras $\widetilde{\cH}$ \eqref{eq:def-cH} and $\widetilde{\cK}$ \eqref{eq:def-cK}.
There are Hopf superalgebra maps $\widetilde{\cH}\hookrightarrow \widetilde{\cD}$ and $\widetilde{\cK}^{\operatorname{op}}\hookrightarrow \widetilde{\cD}$ 
(by the presentation of $\widetilde{\cD}$) and
$\widetilde \cD$ is isomorphic to the double of $\widetilde\cH$ with respect to a suitable skew-pairing, cf. \cite{gzb}.
\end{remark}

\subsection{The double as a super abelian extension}
We show that $\widetilde{\cD}$ fits into an exact sequence of Hopf superalgebras
$R \hookrightarrow \widetilde{\cD} \twoheadrightarrow U$
with $R$ super commutative and $U$ super cocommutative.

Let $\mathfrak{G}$ be the super algebraic group such that its algebra of functions is 
the commutative Hopf superalgebra $\cO(\mathfrak{G}) \coloneqq \ku[X_1,X_2,T^{\pm}]\ot\Lambda(Y_1,Y_2)$ with 
\begin{align*}
|X_1| &= |X_2| = |T| = 0, & |Y_1| &= |Y_2| = 1,
\end{align*} 
and comultiplication
\begin{align}\label{eq:comultiplication-of-super-G}
\begin{aligned}
&\begin{aligned}\Delta(X_1) &= X_1 \ot 1 + T^2 \ot X_1 + Y_1 T \ot Y_1,& \Delta(T) &= T\ot T,\\
\Delta(X_2) &= X_2 \ot 1 + 1\ot X_2 + Y_2\ot Y_2,
\end{aligned}
\\
&\begin{aligned}
\Delta(Y_2) &= Y_2 \ot 1 + 1 \ot Y_2, & \Delta(Y_1) &= Y_1 \ot 1 + T \ot Y_1.
\end{aligned}
\end{aligned}
\end{align}

\begin{theorem}\label{prop:Db-exact}
There is an exact sequence of Hopf superalgebras
\begin{align}\label{eq:Db-exact}
\cO(\mathfrak{G})\hookrightarrow \widetilde{\cD}\twoheadrightarrow U(\mathfrak{osp}(1|2)).
\end{align}
\end{theorem}
\begin{proof}
From the defining relations of $\widetilde{\cD}$ we deduce that
\begin{align*}
\xb_{21} \xb_1 &= \xb_1 \xb_{21}, & \xb_{21} \ub_1 &= \ub_1 \xb_{21}, & \ub_{21} \ub_1 &= \ub_1 \ub_{21}, \\
\gentilde  \xb_{21} &= \xb_{21} \gentilde, & \xb_{21} \ub_{21} &= \ub_{21} \xb_{21}, & \gentilde \ub_{21} &= \ub_{21} \gentilde.
\end{align*}
Hence the  map $\iota: \cO(\mathfrak{G}) \hookrightarrow \widetilde{\cD}$ given by
\begin{align*}
Y_1 &\mapsto \xb_1, & Y_2&\mapsto \ub_1, & X_1 &\mapsto \xb_{21}, & X_2 &\mapsto \ub_{21}, & T&\mapsto \gentilde
\end{align*}
is a well defined injective morphism of Hopf superalgebras.
Next the  map $\pi\colon \widetilde{\cD} \longrightarrow U(\mathfrak{osp}(1|2))$ given by
\begin{align*}
\xb_1 &\mapsto 0, & \xb_2 &\mapsto \psi^-, & \ub_1&\mapsto 0, & \ub_2&\mapsto \psi^+, & \gentilde&\mapsto 1, & \zeta&\mapsto -h.
\end{align*}
is a well defined surjective morphism of Hopf superalgebras,
since $\pi(\xb_{21}) = \pi(\ub_{21})=0$, $\pi(-\xb_2^2) = f$ and $\pi(\ub_2^2) = e$. 
We check that $\ker \pi = \widetilde{\cD} \iota(\cO(\mathfrak{G}))^+$ using the PBW-basis of $\widetilde{\cD}$.
Thus \eqref{eq:Db-exact} is exact.
\end{proof}
\subsection{A central Hopf subalgebra} In this Subsection $\car \ku = p > 2$. We show that $\widetilde{\cD}$ has a central Hopf subalgebra $Z = \cO(\Bb)$, with $\Bb$ a solvable algebraic group. We shall need the following commutation relations in $\widetilde{\cD}$.
\begin{lemma}\label{lemma:commutation-relations-Db}
The following equalities are valid for all $n,m\in\N_0$:
\begin{align*}
\xb_2^{2m} \xb_{21}^n =& \sum_{k=0}^{m} \binom{m}{k} [n]^{[k]} \xb_{21}^{n+k} \xb_2^{2(m-k)},\\
\xb_2^{2m+1} \xb_{21}^n =& \sum_{\ell=0}^1 \sum_{k=0}^{m} \binom{m}{k} [n]^{[k+\ell]} \xb_1^\ell \xb_{21}^{n+k} \xb_2^{2(m-k) -\ell + 1}, \\
\xb_2^{2m} \xb_1 =& \sum_{k=0}^m (-1)^k [-m]^{[k]} \xb_1 \xb_{21}^k \xb_2^{2(m-k)},\\
\xb_2^{2m+1} \xb_1 =& \sum_{\ell=0}^1\sum_{k=0}^m (-1)^{k+\ell} [-m]^{[k]} x_1^\ell \xb_{21}^{k-\ell + 1} \xb_2^{2(m-k)+\ell}\\
\gentilde^n \xb_2^{2m} =& \sum_{k=0}^m \binom{m}{k} [-n]^{[k]} \xb_{21}^k \xb_2^{2(m-k)} \gentilde^n,\\
\gentilde^n \xb_2^{2m+1} =&  \sum_{\ell=0}^1 \sum_{k=0}^m \binom{m}{k} (-1)^{n+1} [-n]^{[k+\ell]} x_1^\ell \xb_{21}^k \xb_2^{2(m-k)-\ell+1} \gentilde^n,\\
\xb_2^{2m} \xb_{21}^n =& \sum_{k=0}^{m} \binom{m}{k} [n]^{[k]} \xb_{21}^{n+k} \xb_2^{2(m-k)},\\
\xb_2^{2m+1} \xb_{21}^n =& \sum_{\ell=0}^1 \sum_{k=0}^{m} \binom{m}{k} [n]^{[k+\ell]} \xb_1^\ell \xb_{21}^{n+k} \xb_2^{2(m-k) -\ell + 1}, \\
\zeta^n \xb_2^m =& \sum_{\ell=0}^{n} \binom{n}{\ell}m^{n-\ell} \xb_2^m \zeta^\ell,\\
\zeta^n \xb_{21}^m =& \sum_{\ell=0}^{n} \binom{n}{\ell}(2m)^{n-\ell} \xb_{21}^m \zeta^\ell,\\
\ub_2^{2m} \ub_1 =& \sum_{k=0}^m (-1)^k [-m]^{[k]} \ub_1 \ub_{21}^k \ub_2^{2(m-k)},\\
\ub_2^{2m+1} \ub_1 =& \sum_{\ell=0}^1\sum_{k=0}^m (-1)^{k+\ell} [-m]^{[k]} x_1^\ell \ub_{21}^{k-\ell + 1} \ub_2^{2(m-k)+\ell}\\
\ub_2^{2m} \gentilde^n  =& \sum_{k=0}^m \binom{m}{k} [-n]^{[k]}  \gentilde^n \ub_{21}^k \ub_2^{2(m-k)},\\
\ub_2^{2m+1} \gentilde^n  =&  \sum_{\ell=0}^1 \sum_{k=0}^m \binom{m}{k} (-1)^{n+1} [-n]^{[k+\ell]} x_1^\ell \gentilde^n \ub_{21}^k \ub_2^{2(m-k)-\ell+1},\\
\ub_2^m \zeta^n  =& \sum_{\ell=0}^{n} \binom{n}{\ell}m^{n-\ell} \zeta^\ell \ub_2^m ,\\
\ub_{21}^m \zeta^n  =& \sum_{\ell=0}^{n} \binom{n}{\ell}(2m)^{n-\ell} \zeta^\ell \ub_{21}^m,\\
\ub_2 \xb_{21}^n =& \xb_{21} \ub_2 - 2n \xb_{21}^n \ub_1 + n \xb_1 \xb_{21}^{n-1}(1+\gentilde),\\
\ub_{21}^n \xb_2 =& \xb_2 \ub_{21}^n + n (\gentilde +1) \ub_1 \ub_{21}^{n-1},\\
\ub_2^{2n} \xb_1 =& \xb_1 \ub_2^{2n} - n \xb_1 \ub_{21} \ub_2^{2(n-1)} + n \ub_1 \ub_2^{2n-1},\\
\ub_2^{2n} \xb_2 =& \xb_2 \ub_2^{2n} - n \xb_2 \ub_{21} \ub_2^{2(n-1)} + n \gentilde \ub_2^{2n-1} \\
&- n\gentilde \ub_1 \ub_2^{2(n-1)} - n(n-1)
\gentilde \ub_{21} \ub_2^{2(n-1)-1},\\
\ub_1 \xb_2^{2n} =& \xb_2^{2n}\ub_1 + n\xb_2^{2(n-1)}\xb_1\gentilde - n(n-1) \xb_{21} \xb_2^{2(n-2)}\xb_1\gentilde,\\
\ub_2 \xb_2^{2n} =& \xb_2^{2n} \ub_2 - 2n \xb_2^{2n}\ub_1 + n \xb_2^{2n-1} - n \xb_2^{2(n-1)}\xb_1 \gentilde \zeta\\
& + n(n-1)\xb_{21}\xb_2^{2(n-2)}\xb_1\gentilde\zeta - n(2n-1)\xb_2^{2(n-1)}\xb_1\gentilde\\
& + n(n-1)(2n-1) \xb_{21}\xb_2^{2(n-2)}\xb_1\gentilde,
\end{align*}
\begin{align*}
\gentilde^n \xb_1 &=  \xb_1 \gentilde^n,& \gentilde^n \xb_{21}^m &= \xb_{21}^m \gentilde^n,& \xb_1 \xb_{21}^n &= \xb_{21}^n \xb_1,\\
\ub_{21}^n \xb_1 &= \xb_1 \ub_{21}^n, & \xb_{21}^n \ub_1 &= \ub_1 \xb_{21}^n.  
\end{align*}
\end{lemma}
\begin{proof}
Straightforward by induction.
\end{proof}
For our next statement we need to set up the notation. 
Let
\begin{align}\label{eq:alggroup-B}
\Bb \coloneqq \left((\Gb_a \times \Gb_a) \rtimes \Gb_m\right) \times \Hb_3
\end{align}
be the algebraic group that in the first factor has the semidirect product where $\Gb_m$ acts on $\Gb_a \times \Gb_a$
by 
$\lambda_ \cdot(r_1, r_2) =(\lambda^2 r_1, \lambda^2 r_2)$, $\lambda \in \kut$, $r_1, r_2\in\ku$ while in the second factor
appears the Heisenberg group $\Hb_3$ i.e. the group of upper triangular matrices with ones in the diagonal.
Let $\zeta^{(p)}\coloneqq \zeta^p-\zeta$ and $Z \coloneqq \ku \langle \xb_{21}^p, \xb_2^{2p}, \ub_{21}^p, \ub_2^{2p}, \gentilde^p,\zeta^{(p)}\rangle \hookrightarrow \widetilde{\cD}$. Note that $Z$ is an even subalgebra.

\begin{theorem}\label{thm:Z} \begin{enumerate}[leftmargin=*,label=\rm{(\roman*)}]
\item\label{Z-1} $Z$ is a central Hopf subalgebra of $\widetilde{\cD}$.

\smallbreak
\item\label{Z-2} $\widetilde{\cD}$ is a finitely generated free $Z$-module.

\smallbreak
\item\label{Z-3} $Z\simeq \ku[T^{\pm},X_1,\dots,X_5]$ as an algebra. In particular $Z$ is a domain.

\smallbreak
\item\label{Z-4} $Z\simeq \cO(\Bb)$ as Hopf algebras.
\end{enumerate}
\end{theorem}
\begin{proof}
\ref{Z-1} By Lemma \ref{lemma:commutation-relations-Db} $Z$ is a even central subalgebra of $\widetilde{\cD}$. We need to verify that is a subcoalgebra and invariant by the antipode. We have the following comultiplication formulas for every $n\in \N$
\begin{align*}
\Delta(\xb_{21}^n) =& \sum_{\ell=0}^{1}\sum_{k=0}^{n-\ell}\binom{n-\ell}{k} n^\ell
\xb_1^\ell \xb_{21}^k \gentilde^{2(n-k)-\ell}\ot \xb_1^\ell 
\xb_{21}^{n-k-\ell},\\
\Delta(\xb_2^{2n}) =& \sum_{\ell=0}^1\sum_{k=0}^{n-\ell} \sum_{t=0}^{k} \binom{n-\ell}{k}\binom{k}{t} n^\ell
[k-n+\ell]^{[t]}\\
&
\xb_1^\ell \xb_{21}^t \xb_2^{2(k-t)} \gentilde^{2(n-k)-\ell}\ot \xb_2^{2(n-k)-\ell},\\
\Delta(\ub_{21}^n) =& \sum_{\ell=0}^{1}\sum_{k=0}^{n-\ell}\binom{n-\ell}{k} n^\ell
\ub_1^\ell \ub_{21}^k \ot \ub_1^\ell \ub_{21}^{n-k-\ell}.
\end{align*}
Hence $\ub_{21}^p,\zeta^{(p)}\in\Pc(\widetilde{\cD})$, $\xb_2^{2p},\xb_{21}^p$ are $(1,\gentilde^{2p})$-primitive and $\gentilde^{p}\in G(\widetilde{\cD})$. It only remains to calculate $\Delta(\ub_2^{2p})$. Recall the formula \eqref{eq:coproduct-dual-super-jordan} and
\begin{align*}
\delta(\ub_2^{2n}) &=\sum_{k=0}^n \sum_{j=0}^{k} \binom{n}{k} \stirling{k}{j} (-1)^k \zeta^j\ot \ub_{21}^k\ub_2^{2(n-k)}, & n\in\N,
\end{align*}
where $\stirling{k}{j}$ are the Stirling numbers. Since $[\zeta]^{[p]} = \prod_{i=1}^{p} (\zeta+i-1) = \sum_{k=0}^{p} \stirling{p}{k} \zeta^k = \zeta^p - \zeta$, we have 
\begin{align}\label{eq:stirling}
\stirling{p}{k} &= 0, \ k=2,\dots,p-1, &
\stirling{p}{p}&=1, &\text{ and }\stirling{p}{1} &= -1.
\end{align} 
Then we get $\Delta(\ub_2^{2p}) = \ub_2^{2p}\ot 1 + 1\ot \ub_2^{2p} - \zeta^{(p)}\ot \ub_{21}^p$.

\ref{Z-2} To prove this we consider another basis of $\widetilde{\cD}$ using a different basis of $\ku[\zeta]$. The family of
polynomials 
\begin{align*}
(\zeta^{(p)})^k \zeta^j, && k\in\N_0, j\in\I_{0, p-1},
\end{align*}
is a basis of $\ku[\zeta]$, see the proof of \cite[Prop. 2.6]{ap}. Thus the elements
\begin{align*}
\xb_1^{n_1} \xb_{21}^{n_{21}} \xb_2^{n_2} \gentilde^n (\zeta^{(p)})^k \zeta^j \ub_1^{m_1} \ub_{21}^{m_{21}} \ub_2^{m_2},
\end{align*}
with $(n_1,m_1,n,n_2,n_{21},k,m_2,m_{21},j)\in \I_{0,1}^2\times \N_0^6\times \I_{0, p-1}$ form a basis of $\widetilde{\cD}$. Hence a basis of $\widetilde{\cD}$ as a $Z$-module is given by
\begin{align*}
\xb_1^{n_1} \xb_{21}^{n_{21}} \xb_2^{n_2} \gentilde^n \zeta^j \ub_1^{m_1} \ub_{21}^{m_{21}} \ub_2^{m_2},
\end{align*}
with $(n_1,m_1,n_2,m_2,n_{21},m_{21},j,n)\in \I_{0, 1}^2\times\I_{0, 2p-1}^2\times\I_{0, p-1}^4$.

\ref{Z-3} The map $\phi\colon \ku[T^{\pm},X_1,\dots,X_5] \longrightarrow Z$ given by
\begin{align*}
T&\mapsto \gentilde^p, & X_1&\mapsto \xb_{21}^p, & X_2&\mapsto \xb_2^{2p}, & X_3&\mapsto -\zeta^{(p)}, & X_4&\mapsto \ub_{21}^{p}, &
X_5&\mapsto \ub_2^{2p},
\end{align*}
is the desired isomorphism of algebras.

\ref{Z-4} It is easy to see that $\cO(\Bb)\simeq \ku[T^{\pm},X_1,\dots,X_5]$ with comultiplication determined by $T\in G(\cO(\Bb))$, $X_1, X_2 \in\Pc_{T^2,1}(\cO(\Bb))$,
$X_3, X_5\in\Pc(\cO(\Bb))$ and
$\Delta(X_5) = X_5\ot 1 + 1\ot X_5 + X_3\ot X_4$.
The claim follows.
\end{proof}

\subsection{Ring theoretical properties of the double}
\begin{prop}
	\begin{enumerate}[leftmargin=*,label=\rm{(\roman*)}]
\item\label{prop:Ring-Prop-1} The algebra $\widetilde{\cD}$ admits an exhaustive ascending filtration
$(\widetilde{\cD}_n)_{n\in\N_{0}}$ such that $\gr \widetilde{\cD} \simeq \ku[X_1,\dots,X_5,T^{\pm 1}] \ot \Lambda(Y_1,\dots,Y_4)$.
\item\label{prop:Ring-Prop-2} The algebras $\hopfdoble$ and $\widetilde{\cD}$ are noetherian.
\item\label{prop:Ring-Prop-3} If $\car \ku = p > 2$ then $\widetilde{\cD}$ is a PI-algebra.
	\end{enumerate}
\end{prop}
\begin{proof}
 \ref{prop:Ring-Prop-1} Let $A$ be the algebra presented by generators $\xb_1$, $\xb_2$, $\ub_1$, $\ub_2$, $\zeta$, $\gentilde^{\pm1}$ and $Z_i$, $i \in \I_4$, with relations $\gentilde^{\pm 1} \gentilde^{\mp 1} = 1$. 
 The algebra  $A$ is graded with
 \begin{align*}
 \deg \xb_1 &= \deg \ub_1 = 3, &  \deg \xb_2 &= \deg \ub_2 = 4, & \deg \gentilde^{\pm 1} &= \pm 1,\\
 \deg Z_1 &= \deg Z_2 = \deg \zeta = 1, & \deg Z_3 &= \deg Z_4 = 2. 
 \end{align*}
 The filtration associated to this grading induces a filtration on $\widetilde{\cD}$ via the epimorphism $A\twoheadrightarrow\widetilde{\cD}$ given by
\begin{align*}
Z_1 &\mapsto \xb_{21}, & Z_2 &\mapsto \ub_{21}, & Z_3 &\mapsto \xb_2^2, & Z_4 &\mapsto \ub_2^2,
\end{align*}
the remaining generators being mapped to their homonyms. 
The relations of $\widetilde{\cD}$ imply that $\gr \widetilde{\cD}$ is super conmutative with the same parity as $\widetilde{\cD}$. Then $\phi\colon \ku[X_1,\dots,X_5,T^{\pm 1}] \ot \Lambda(Y_1,\dots,Y_4) \rightarrow \gr \widetilde{\cD}$ given by $X_i\mapsto Z_i$, $i \in \I_4$, and
\begin{align*}
T&\mapsto \gentilde, & Y_1&\mapsto \xb_1, & Y_2&\mapsto \xb_2, & Y_3&\mapsto \ub_1, & Y_4&\mapsto \ub_2, & X_5&\mapsto \zeta
\end{align*}
is an isomorphism of algebras by comparison of the Hilbert series.

\ref{prop:Ring-Prop-2} It is well-known that $\ku[X_1,\dots,X_5,T^{\pm 1}] \ot \Lambda(Y_1,\dots,Y_4)$ is noetherian,
hence so is $\widetilde{\cD}$ by
\ref{prop:Ring-Prop-1} and a fortiori $\hopfdoble$ which is a finitely generated $\widetilde{\cD}$-module.
\ref{prop:Ring-Prop-3} follows from Theorem \ref{thm:Z}, Proposition \ref{prop:cD-as-Hopf-superalgebra} and \cite[Corollary 1.13]{McRob}. 
\end{proof}

\section{The double of the restricted super Jordan plane} \label{sec:double-restricted-superjordan}
In this Section, $\car \ku  = p > 2$. 

\subsection{The bosonizations}
Recall that $\polring_p = \ku[\zeta]/(\zeta^p- \zeta)$ is a Hopf algebra
with $\zeta$ primitive.
Besides \eqref{eq:generators-groups} we also fix the generators 
\begin{align}\label{eq:generators-groups-p}
C_p &= \langle g \rangle, & C_{2p} &= \langle \gentwop \rangle.
\end{align}
It is well-known that $\ku^{C_p} \simeq \polring_p$, see e.g. \cite[1.3]{ap}. 
Hence
$\ku^{C_{2p}} \simeq \polring_p \ot \ku C_2$ and
the algebra $\ku^{C_{2p}}$ is presented  by generators $\genpi$ and $\zeta$ with relations
\begin{align}\label{eq:relations-super-k^gamma}
\genpi^2 &= 1, 
\\\label{eq:relaciones-K-cursiva}
\zeta^p &= \zeta, 
\\\label{eq:relations-super-k^gamma2}
 \genpi \zeta &= \zeta \genpi.
\end{align}

We consider the realizations of $V$ in $\customyd{\ku C_{2p}}$ and $W$ in $\customyd{\ku^{C_{2p}}}$
given by
\begin{align}\label{eq:V_as_YD_module}
&\begin{aligned}
\gentwop\rightharpoonup x_1 &= -x_1, & \gentwop\rightharpoonup x_2 &= -x_2 + x_1,&
\delta(x_i) &= \gentwop\ot x_i,& i=1,2;
\end{aligned}
\\
\label{eq:W_as_YD_module}
&\begin{aligned}
\zeta\rightharpoonup u_i &= u_i,& \genpi\rightharpoonup u_i &= -u_i,& i &=1,2,
\\
\delta(u_1) &= \genpi\ot u_1, & \delta(u_2) &= \genpi\ot u_2-\zeta\genpi\ot u_1, &&
\end{aligned}
\end{align}
Thus we have the Hopf algebras 
\begin{align*}
H &\coloneqq \toba(V)\#\ku C_{2p}, & K &\coloneqq \toba(W) \#\ku^{C_{2p}}.
\end{align*}
They have PBW-basis, denoted by $\mathsf{B}_{H}$ and $\mathsf{B}_{K}$,  given by the ordered monomials \eqref{eq:monomials:jordan}
or \eqref{eq:monomials:jordan-dual} with  $(n_1, n_{21}, n_2) \in\I_{0, 1} \times \N_{0} \times \N_{0}$,
multiplied accordingly by elements of the groups $C_p$, $C_{2p}$, or powers of $\zeta$. Then
\begin{align*}
\dim H &= \dim K = 8p^3.
\end{align*}
The comultiplications  are given by $x_1, x_2 \in \Pc_{\gamma, 1}(H)$,
$u_1 \in \Pc_{\epsilon, 1}(K)$ and $\Delta_{K}(\ub_2)$ by \eqref{eq:comultiplication-Hb1} 
Therefore there are surjective Hopf algebra maps
\begin{align*}
\hopfuno &\twoheadrightarrow H, & \hopfdos &\twoheadrightarrow K.
\end{align*}
The basic properties of $H$ and $K\simeq H^*$ follow without difficulties.

\subsection{The double} 
In this subsection we show that the Drinfeld double  of $H$ fits into an exact sequence of Hopf algebras
$\ku C_2 \hookrightarrow D(H) \twoheadrightarrow D$.

\medbreak
We first give a presentation of $D(H)$; for this we need that of $D(\ku C_{2p})$
which follows easily since $D(\ku C_{2p})\simeq \ku C_{2p}\ot \ku^{C_{2p}} \simeq \ku C_{2p}\ot \polring_p \ot \ku C_2$.

\begin{prop}\label{prop:double-H}
$D(H)$ is generated by  $x_1, x_2, \zeta, \genpi, \gentwop, u_1, u_2$ 
with relations \eqref{eq:def-super-jordan}, \eqref{eq:def-super-jordan-restricted}, \eqref{eq:relations-H-dual2}, \eqref{eq:genpi-xb}, \eqref{eq:zeta-xi}, \eqref{eq:relations-super-k^gamma}, \eqref{eq:relaciones-K-cursiva}, \eqref{eq:relations-super-k^gamma2} and
\begin{align}
\label{relations-super-double-kgamma}
\gentwop^{2p} &= 1, & \genpi \gentwop &= \gentwop \genpi, \quad \zeta \gentwop = \gentwop \zeta,
\\\label{eq:relations-H}
\gentwop x_1  &= -x_1\gentwop, & \gentwop x_2 &= (-x_2 + x_1) \gentwop, 
\\
\label{eq:relations-H-dual-op}
u_1^2 &= 0,  &u_2 u_{21} &= u_{21} u_2 + u_1 u_{21},
\\
\label{eq:relations-H-dual3}
u_i \zeta  &= \zeta u_i  + u_i, & i  &=  1, 2.
\\
\label{eq:relation-DH3}
u_1 x_1 &= -x_1 u_1, & u_1 x_2 &= -x_2 u_1 +(1-\gentwop \genpi),
\\ \label{eq:relation-DH4}
u_2 x_1 &= -x_1 u_2 +(1-\gentwop\genpi) + x_1 u_1, &  u_2 x_2 &= -x_2 u_2 +\gentwop \zeta \genpi + x_2 u_1, 
\\ \label{eq:relation-DH5}
u_1 \gentwop &= -\gentwop u_1, &
u_2 \gentwop &=-\gentwop u_2 + \gentwop u_1. 
\end{align} 
The comultiplication is determined by $\gentwop, \genpi \in G(D(H))$, $\zeta \in \Pc(D(H))$, 
$x_1, x_2 \in \Pc_{\gamma, 1}(D(H))$, $u_1 \in \Pc_{\genpi, 1}(D(H))$ and \eqref{eq:comultiplication-Hb1}. The monomials
\begin{align*}
x_1^{n_1} x_{21}^{n_{21}} x_2^{n_2} \gentwop^i \genpi^j \zeta^k u_1^{m_1} u_{21}^{m_{21}} u_2^{m_2}
\end{align*}
with $(n_1,m_1,j,n_{21},k,m_{21},i,n_2,m_2)\in \I_{0, 1}^3\times \I_{0, p-1}^3\times\I_{0, 2p-1}^3$ is a PBW-basis of $D(H)$.
\end{prop}
\begin{proof}
This is a direct application of  Proposition \ref{prop:the-double-in-general}.
\end{proof}

Let $t \coloneqq \gentwop^p \genpi$ and $g \coloneqq \gentwop^{p+1}$. The Hopf subalgebra generated by $t$ is isomorphic to $\ku C_2$ and the one generated by $g$ to $\ku C_p$. By the defining relations $t$ is a central element. 
Let $Z_0 = \ku \langle t \rangle$ and $D \coloneqq D(H)/D(H)Z_0^+$. We use the same symbol for an element in $D(H)$ and its class in $D$.
\begin{prop}\label{prop:definition-of-cD}
\begin{enumerate}[leftmargin=*,label=\rm{(\alph*)}]
\item\label{af: algebra-D-1} The  algebra $D$ is generated by $x_1, x_2, g, \zeta, u_1, u_2, \genpi$ with relations  \eqref{eq:def-super-jordan}, \eqref{eq:def-super-jordan-restricted}, \eqref{eq:relations-H-dual2}, \eqref{eq:genpi-xb}, \eqref{eq:zeta-xi}, \eqref{eq:relations-H-dual-op}, \eqref{eq:relations-H-dual3}, \eqref{eq:relations-super-k^gamma}, \eqref{eq:relaciones-K-cursiva}, \eqref{eq:relations-super-k^gamma2}
\begin{align}\label{eq:relations-D1}
 \zeta g &= g \zeta, & g^{p} &= 1,
\\
\label{eq:relations-sH}
gx_1  &= x_1g, & g x_2 &=  (x_2  - x_1) g,
\\\label{eq:relations-D3}
u_2 g &= g u_2 - g u_1, & u_1 g &= g u_1, 
\\\label{eq:relations-D4}
u_1 x_2 &= -x_2 u_1 +(1-g), &  u_2 x_2 &= -x_2 u_2 +g \zeta + x_2 u_1, & & 
\\ \label{eq:relations-D5}
u_2 x_1 &= -x_1 u_2 +(1-g) + x_1 u_1, & u_1 x_1 &= -x_1 u_1, &  \\\label{eq:relations-D6}
\genpi g&= g \genpi.
\end{align}

\item\label{af: algebra-D-2} The sequence of Hopf algebras $Z_0 := \ku C_2 \hookrightarrow D(H) \twoheadrightarrow D$
is exact.

\item\label{af: algebra-D-3}The algebra $D$ has dimension $32 p^6$ and basis consisting in the monomials
\begin{align*}
x_1^{n_1} x_{21}^{n_{21}} x_2^{n_2} g^n \zeta^m u_1^{m_1} u_{21}^{m_{21}} u_2^{m_2} \genpi^k 
\end{align*}
with $(n_1,m_1,k,n_{21},m,m_{21},n,n_2,m_2)\in \I_{0, 1}^3\times \I_{0, p-1}^4\times\I_{0, 2p-1}^2$
\end{enumerate}
\end{prop}

\begin{proof}
\ref{af: algebra-D-1}: Very similar to the proof of Proposition \ref{prop:presentation-tilde-D-0}.
\ref{af: algebra-D-2} Conditions \ref{suc-exacta-1} and \ref{suc-exacta-2} above
are evident while \ref{suc-exacta-3} is proved considering the basis of $D(H)$
\begin{align*}
x_1^{n_1} x_{21}^{n_{21}} x_2^{n_2} g^n \zeta^m u_1^{m_1} u_{21}^{m_{21}} u_2^{m_2} \genpi^k (1-t)^\ell 
\end{align*}
with $(n_1,m_1,\ell,k,n,m,n_{21},m_{21},n_2,m_2)\in\I_{0, 1}^4\times\I_{0, p-1}^4\times\I_{0, 2p-1}^2$.
Now  Remark \ref{remark-exact-sequence-hopf} gives \ref{suc-exacta-4} since $Z_0$ is central. 
The proof of \ref{af: algebra-D-3} is direct. 
\end{proof} 

\begin{remark}\label{rem:D-doble} 
	There is an exact sequence of Hopf algebras
	$\cO(\Bb) \xhookrightarrow[]{} \hopfdoble \xrightarrowdbl[]{} D$.
\end{remark}

\subsection{The Hopf superalgebra $\cD$}
Here we show that 
$D$ is the bosonization of the Hopf superalgebra
$\cD := \ku \langle x_1, x_2, u_1, u_2, g,\zeta\rangle \hookrightarrow D$ and 
that $\cD$ is a restricted analogue of the Hopf superalgebra $\widetilde{\cD}$.

\begin{prop}\label{prop:cD-as-Hopf-superalgebra}
\begin{enumerate}[leftmargin=*,label=\rm{(\roman*)}]
\item \label{prop:propiedades-cD-5} A PBW-basis of $\cD$ is given by the monomials
\begin{align*}
x_1^{n_1} x_{21}^{n_{21}} x_2^{n_2} g^n \zeta^m u_1^{m_1} u_{21}^{m_{21}} u_2^{m_2},
\end{align*}
\end{enumerate}
\noindent $(n_1,m_1,n,m,n_{21},m_{21},n_2,m_2)\in\I_{0, 1}^2\times\I_{0, p-1}^4\times\I_{0, 2p-1}^2$. Thus $\dim \cD =16 p^6$.

\medbreak\begin{enumerate}[leftmargin=*,label=\rm{(\roman*)},resume]
\item\label{prop:propiedades-cD-6} $\cD$ has a triangular decomposition i.e. a linear isomorphism induced by multiplication  
\begin{align}\label{eq:D-cursiva-triang-decomp}
\cD \simeq \cB(V) \ot (\ku C_p\ot \polring_p) \ot \cB(W)^{\operatorname{op}}.
\end{align}

\medbreak
\item\label{prop:propiedades-cD-4} $\cD$ is generated by $x_1$, $x_2$, $u_1$, $u_2$, $g$, $\zeta$ with relations \eqref{eq:def-super-jordan}, \eqref{eq:def-super-jordan-restricted}, \eqref{eq:relations-H-dual2}, \eqref{eq:zeta-xi}, \eqref{eq:relaciones-K-cursiva}, \eqref{eq:relations-sH},  \eqref{eq:relations-H-dual-op}, \eqref{eq:relations-H-dual3}, \eqref{eq:relations-D1}, \eqref{eq:relations-D3}, \eqref{eq:relations-D4}, \eqref{eq:relations-D5}.

\medbreak
\item\label{prop:propiedades-cD-1} $\cD$ is a Hopf superalgebra with comultiplication 
given by $g \in G(\cD)$, $\zeta, u_1 \in \Pc(\cD)$, 
$x_1, x_2 \in \Pc_{g, 1}(\cD)$ and \eqref{eq:comultiplication-Dtilde}.
Indeed $\cD\# \ku C_2 \simeq D$ as Hopf algebras. 

\medbreak
\item\label{prop:propiedades-cD-2} There is an exact sequence of Hopf superalgebras
\begin{align*}
\cO(\Bb) \xhookrightarrow[]{\iota} \widetilde{\cD}  \xrightarrowdbl[]{\pi} \cD.
\end{align*}
\end{enumerate}
\end{prop}

\begin{proof}
\ref{prop:propiedades-cD-5}, \ref{prop:propiedades-cD-6}, \ref{prop:propiedades-cD-4} are analogous to the ones of
Proposition \ref{prop:Db}.
\ref{prop:propiedades-cD-1} As in Proposition \ref{prop:Db}, show that $\cD \simeq D^{\operatorname{co} \varpi}$
where $\varpi\colon D \twoheadrightarrow \ku C_2$ is given by
\begin{align*}
x_1&\mapsto 0, & x_2&\mapsto 0, & u_1&\mapsto 0, & u_2&\mapsto 0, & g&\mapsto 1, & \zeta&\mapsto 0, & \genpi &\mapsto \genpi.
\end{align*}

\ref{prop:propiedades-cD-2} By the presentations of $\widetilde{\cD}$ and $\cD$ there exist $\pi\colon \widetilde{\cD}\twoheadrightarrow \cD$ such that
\begin{align*}
\gen &\mapsto g, & \zeta &\mapsto\zeta, & \xb_i&\mapsto x_i, & \ub_i&\mapsto u_i, & i&=1,2.
\end{align*}
The map $\iota\colon \cO(\Bb) \hookrightarrow \widetilde{\cD}$ is defined as the $\phi$ in the proof of Theorem \ref{thm:Z}. Clearly 
$\ker \pi =\widetilde{\cD}\iota(\cO(\Bb))^+$ by the PBW bases. Since $\iota(\cO(\Bb))$ is central and $\widetilde{\cD}$ is a free module over $\iota(\cO(\Bb))$, the claim follows from Remark \ref{remark-exact-sequence-hopf}.
\end{proof}

\begin{remark}
We may realize $V$ in $\customyds{\ku C_p}$ and $W$ 
in $\customyds{\polring_p}$ by
\begin{align*}
&\begin{aligned}
g\rightharpoonup x_1 &= x_1, \ g\rightharpoonup x_2 = x_2 - x_1, \
\delta(x_i) = g \ot x_i, \ |x_i| = 1, \ i=1,2;
\end{aligned}
\\ 
&\begin{aligned}
\zeta\rightharpoonup\ub_i &= \ub_i,& |u_i| &= 1,& i &=1,2,
\\
\delta(\ub_1) &= 1\ot \ub_1, & \delta(\ub_2) &= 1\ot \ub_2 - \zeta\ot \ub_1.&&
\end{aligned}
\end{align*}
Thus we have the Hopf superalgebras
$\cH \coloneqq \toba(V) \supbos \ku C_{p}$, $\cK \coloneqq \toba(W) \supbos \polring_p$; clearly  
$\dim \cH = \dim \cK = 4p^3$.
Their comultiplications  are determined by $x_1, x_2 \in \Pc_{g, 1}(\cH)$,
$u_1\in \Pc(\cK)$, while $\Delta_{\cK}(\ub_2)$ is given by \eqref{eq:comultiplication-Dtilde}.
Therefore there are surjective Hopf algebra maps
$\widetilde{\cH} \twoheadrightarrow \cH$, $\widetilde{\cK}\twoheadrightarrow \cK$.
The basic properties of $\cH$ and $\cK$ follow at once.
Also, there are isomorphisms of Hopf algebras
$H \simeq  \cH \#\ku C_2$, $K \simeq \cK \# \ku C_2$.
Finally, there are morphisms of Hopf superalgebras  $\cH\hookrightarrow \cD$, $\cK^{\operatorname{op}}\hookrightarrow \cD$,
hence $\cD \simeq D(\cH)$, see \cite{gzb}.
\end{remark}

\subsection{$\cD$ as an extension}
We next show that $\cD$ fits into an exact sequence of Hopf superalgebras
$\nucleo \hookrightarrow \cD \twoheadrightarrow \ugo$
with $\nucleo$  super commutative  and $\ugo$ super cocommutative.
First let $\nucleo$ be the super commutative  Hopf superalgebra
\begin{align*}
\nucleo \coloneqq \ku[X_1, X_2, T]/(X_1^p, X_2^p, T^p-1) \ot \Lambda(Y_1,Y_2)
\end{align*} 
with $|X_1| = |X_2| = |T| = 0$, $|Y_1| = |Y_2| = 1$ and comultiplication \eqref{eq:comultiplication-of-super-G}. 
Arguing as in Proposition \ref{eq:Db-exact}, we have:
\begin{theorem}\label{prop:cD-as-super-abelian-extension}
There exist Hopf superalgebra maps $\iota$ and $\pi$ such that  
\begin{align}\label{eq:cD-exact}
\nucleo \xhookrightarrow[]{\iota} \cD  \xrightarrowdbl[]{\pi} \ugo(\mathfrak{osp}(1|2))
\end{align}
is an exact sequence of Hopf superalgebras. \qed
\end{theorem}

The simple $\cD$-supermodules can be determined from the  previous result.

\begin{theorem}\label{th:irrep-restricted-super-jordan}
There are exactly $p$ isomorphism classes of simple $\cD$-mo\-dules which have dimensions $1,3,5,\dots,2p-1$. 
\end{theorem}

\pf Being nilpotent, the  ideal $\cD\nucleo ^+ = \langle x_1, x_{21}, u_1, u_{21}, g-1 \rangle$
is contained in the Jacobson radical of $\cD$; thus $\Irr \cD \simeq \Irr \ugo(\mathfrak{osp}(1|2))$
and \cite[Prop. 6.3]{wz} applies.
\epf

\subsection{Simple modules}
We describe the simple $\cD$-modules
as quotients of Verma modules reproving Theorem \ref{th:irrep-restricted-super-jordan}. 
Let $\cD = \oplus_{n\in\Z}\cD^n$ be $\Z$-graded by 
\begin{align*}
\deg x_1 =\deg x_2 = -1, && \deg u_1 =\deg u_2 = 1, && \deg g = \deg \zeta = 0.
\end{align*}
Recall that $D(\ku C_p) \simeq \polring_p \otimes \ku C_p$.
Consider the triangular decomposition  \eqref{eq:D-cursiva-triang-decomp} 
and the  graded subalgebras
$\D^{> 0} \coloneqq \toba(W)^{\operatorname{op}}$, $\D^{< 0}\coloneqq \toba(V)$.
Then
\begin{enumerate}
\smallbreak\item $\D^{> 0}\subseteq\oplus_{n\in\N_0} \cD^n$, $\D^{< 0}\subseteq\oplus_{n\in -\N_0} \cD^n$ and $D(\ku C_p)\subseteq \cD^0$.

\smallbreak
\item $(\D^{> 0})^0 = \ku = (\D^{< 0})^0$.

\smallbreak
\item $\D^{\geq0} := D(\ku C_p)\D^{> 0}$ and $\D^{\leq0} := \D^{< 0}D(\ku C_p)$ are subalgebras of $\cD$.
\end{enumerate}

In this context the simple modules of $\cD$ arise inducing from $\D^{\geq0}$. 
The elements of $\Lambda := \Irr D(\ku C_p)$ are called \emph{weights}. 
Since $\D^{> 0}$ is local, the (homogeneous)
projection $\D^{\geq0} \twoheadrightarrow D(\ku C_p)$ allows to identify
$\Lambda \simeq \Irr \D^{\geq0}$.
The  Verma module associated to $\lambda\in\Lambda$ is
\begin{align*}
\verma(\lambda) = \Ind^{\cD}_{\D^{\geq0}}\lambda = \cD\ot_{\D^{\geq0}} \lambda.
\end{align*}
By a standard argument, $\verma(\lambda)$ is indecomposable. 
Let $L(\lambda)$ be the head of $\verma(\lambda)$. 
The following result is well-known, see for instance \cite[Theorem 2.1]{Vay}.

\begin{lemma} The map $\lambda \mapsto L(\lambda)$ gives a bijection $\Lambda \simeq \Irr \cD$.  \qed
\end{lemma}

The set $\Lambda$ is easy to compute since $D(\ku C_p) \simeq \ku C_p \otimes \polring_p$
and $\ku C_p$ is local. Given $k \in \fp$, let $\lambda_k = \ku w_k$ be the one-dimensional vector space with action
\begin{align*}
g\cdot w_k = w_k, && \zeta\cdot w_k = k w_k.
\end{align*}

\begin{lemma}The map $k \mapsto \lambda_k$ provides a bijection $\fp \simeq \Lambda$.
\qed.
\end{lemma}

We fix $k \in \fp$ and compute $L(\lambda_k)$. 
Since $\verma(\lambda_k)$ is free as a $\D^{< 0}$-module with basis $w_k$, the elements
\begin{align*}
w_k^{(n_1,n_{21},n_2)} &\coloneqq  x_1^{n_1} x_{21}^{n_{21}} x_2^{n_2} \cdot w_k, & n_1\in\I_{0, 1},n_{21}\in \I_{0, p-1},n_2 \in\I_{0, 2p-1},
\end{align*}
form a linear basis of $\verma(\lambda_k)$.
This makes $\verma(\lambda_k)$ a graded module by $\deg w_k^{(n_1,n_{21},n_2)} = \deg (x_1^{n_1} x_{21}^{n_{21}} x_2^{n_2})$, so $\verma(\lambda_k) =\oplus_{n\leq 0}\verma(\lambda_k)_{n} $
and $\verma(\lambda_k)_0 = \ku w_k$. Any proper submodule of $\verma(\lambda_k)$ is then contained in $\oplus_{n\leq -1}\verma(\lambda_k)_{n}$.  Since $L(\lambda_k)$ is the unique simple quotient of $\verma(\lambda_k)$, we divide the later by proper submodules until we get a simple one.
We start by the submodule $N_k$ of $\verma(\lambda_k)$ generated by $S_k \coloneqq \{w_k^{(1,0,0)}, w_k^{(0,1,0)} \}$.

\begin{lemma} The submodule $N_k$ is proper.
\end{lemma}
\begin{proof}
The action of $u_1$ and $u_2$ gives
\begin{align*}
u_1 x_1\cdot w_k &= -x_1 u_1\cdot w_k = 0, \qquad u_1 x_{21}\cdot w_k = x_{21} u_1\cdot w_k = 0,\\
u_2 x_1\cdot w_k &= -x_1 u_2 \cdot w_k + (1-g)\cdot w_k + x_1 u_1\cdot w_k= 0,\\
u_2 x_{21}\cdot w_k &= x_{21} u_2 \cdot w_k - 2 x_{21} u_1\cdot w_k + x_1 (1+g)\cdot w_k= 2 x_1 \cdot w_k.
\end{align*}
So $\D^{> 0}\cdot S_k\subseteq S_k$. Then $N_k = \D^{\leq0}\cdot S_k \subseteq \oplus_{n\leq -1}\verma(\lambda_k)_{n}$
is proper. 
\end{proof}
Let $V_k = \verma(\lambda_k) / N_k$ and let $y_j$ be the class of $w_k^{(0,0,j)}$ in $V_k$.
\begin{lemma}
The family $(y_j)_{j\in\I_{0, 2p-1}}$ generates linearly $V_k$ and $g$, $ u_1$, $u_{21}$, $x_1$ and $x_{21}$ act trivially on $V_k$.
\end{lemma}
\begin{proof}
We claim that the class of $w_k^{(n_1,n_{21},n_2)} = 0$ in $V_k$ if $(n_1,n_{21})\neq (0,0)$.
It suffices to show that $w_k^{(0,1,n_2)} = 0$ and $w_k^{(1,0,n_2)} = 0$  in $V_k$ for every $n_2$. 
This follows by induction on $n_2$ using the analogues for $\cD$ of the relations in Lemma \ref{lemma:commutation-relations-Db}. Since $x_{21}$ and $x_1$ commute, they both act trivially on $V_k$:
\begin{align*}
x_{21}\cdot w_k^{(n_1,n_{21},n_2)} = x_1^{n_1} x_{21}^{n_{21}}\cdot w_k^{(0,1,n_2)} = 0,\\
x_1\cdot  w_k^{(n_1,n_{21},n_2)} = x_1^{n_1} x_{21}^{n_{21}} \cdot w_k^{(1,0,n_2)} = 0.
\end{align*}
Then $(y_j)_{j\in\I_{0, 2p-1}}$ generates linearly $V_k$. 
Also $g$, $u_1$ and $u_{21}$ act trivially on these generators by Lemma \ref{lemma:commutation-relations-Db}.
\end{proof}
Set $y_{-1} = y_{2p} = 0$. The action of $\cD$ on $V_k$ can be computed inductively:
\begin{align}\label{eq:action-Lk}
\begin{aligned}
&\begin{aligned}
\zeta\cdot y_j &= (k-j)y_j, & g\cdot y_j&= y_j, &
x_2\cdot y_j &= y_{j+1}, & x_1\cdot y_j &= 0,
\end{aligned}
\\
&\begin{aligned}
u_1 \cdot y_j &= 0, &u_2 \cdot y_j &= \begin{cases}
\frac{j}{2} y_{j-1} & \text{if } j \text{ is even,}\\
(\frac{j-1}{2}-k) & \text{if } j \text{ is odd,} 
\end{cases}&  j&\in\I_{0, 2p-1}.
\end{aligned}
\end{aligned}
\end{align}

Now $\widetilde{V}_k := \cD y_{2k+1}$ is a proper submodule of $V_k$ because $\D^{> 0}\cdot y_{2k+1} = 0$.

\begin{prop}\label{prop:irrep-cD}
The module $L_k = V_k/ \widetilde{V}_k$ is simple of dimension $2k+1$.
\end{prop}
It follows that $L_k = L(\lambda_k)$, the head of the Verma module $M(\lambda_k)$.

\begin{proof} Let $z_j$ be the class of $y_j$ in $L_k$; the action of $\cD$ on the $z_j$'s is still given by \eqref{eq:action-Lk}.
Then $(z_j)_{j\in\I_{0, 2k}}$ is a basis of $L_k$. 
To see that $L_k$ is simple, we show that every $0 \neq z \in L_k$ generates $L_k$. Let $z = \sum_{j=0}^{m} c_j\,z_j$ 
with $m\leq 2k$ and $c_m\neq 0$. Then $ u_2^m\cdot z\in \kut z_0$, and $\cD\cdot z = L_k$.
\end{proof}

\subsection{$\cO(\mathfrak{G})$ as an extension}
Let $\Gb = (\Gb_a \times \Gb_a) \rtimes \Gb_m$ be the semidirect product where $\Gb_m$ acts on $\Gb_a \times \Gb_a$
by 
$\lambda_ \cdot(t_1, t_2) =(\lambda^2 t_1, t_2)$, $\lambda \in \kut$, $t_1, t_2\in\ku$.
Then $\cO(\Gb)$ is isomorphic to $A := \ku[\mathrm{X_1},\mathrm{X_2}, \mathrm{T}^{\pm 1}]$ with
comultiplication given by $\mathrm{T} \in G(A)$, $\mathrm{X_1} \in \Pc_{\mathrm{T}^2,1}(A)$, $\mathrm{X_2} \in \Pc(A)$.

\begin{prop}\label{prop:cOG-exact}
There is a short exact sequence of Hopf superalgebras
\begin{align}\label{eq:cOG-exact}
\cO(\Gb) \xhookrightarrow[]{\iota} \cO(\mathfrak{G})  \xrightarrowdbl[]{\pi} \nucleo,
\end{align}
with morphisms $\iota\colon \cO(\Gb)\hookrightarrow\cO(\mathfrak{G})$, $\pi\colon\cO(\mathfrak{G})\longrightarrow \nucleo$
given by
\begin{align*}
\iota(\mathrm{X_1}) &= X_1^p, & \iota(\mathrm{X_2}) &= X_2^p, & \iota(\mathrm{T}) &= T^p,\\
\pi(X_1) &= X_1, & \pi(X_2) &= X_2, & \pi(T) &= T, & \pi(Y_1)&= Y_1, & \pi(Y_2)&= Y_2.
\end{align*}
\end{prop}
\begin{proof}
Clearly $\pi$ is of Hopf superalgebras by the definition of comultiplication in $\cO(\mathfrak{G})$ and $\nucleo$. In
 $\cO(\mathfrak{G})$ we have the following comultiplication formulas 
 \begin{align*}
 \Delta(X_1^n) &= \sum_{\ell=0}^{1}\sum_{k=0}^{n-\ell}\binom{n-\ell}{k} n^\ell
 Y_1^\ell X_1^k T^{2(n-k)-\ell}\ot Y_1^\ell 
 X_1^{n-k-\ell},\\
 \Delta(X_2^n) &=\sum_{\ell=0}^{1}\sum_{k=0}^{n-\ell}\binom{n-\ell}{k} n^\ell
 Y_2^\ell X_2^k \ot Y_2^\ell X_2^{n-k-\ell},
 \end{align*}
$n\in\N$. Thus $X_1^p \in \Pc_{\mathrm{T}^{2p},1}(\cO(\mathfrak{G}))$, $X_2^p \in \Pc(\cO(\mathfrak{G}))$
and $\iota$ is of Hopf superalgebras. Then $\ker \pi = \cO(\mathfrak{G}) \iota(\cO(\Gb) )^+$, $\iota$ is injective and $\pi$ is surjective, using the PBW bases. Since  $\cO(\mathfrak{G})$ is a free $\cO(\Gb)$-module, Remark \ref{remark-exact-sequence-hopf} applies.
\end{proof}
\subsection{A commutative square}
We summarize the relationship between the Hopf superalgebras studied so far 
in the commutative diagram
\begin{align}\tag{\ref{eq:diagram-exact sequences}}
\begin{aligned}
\xymatrix{ & \cO(\Gb)  \ar@{^{(}->}[r] \ar@{^{(}->}[d] & \cO(\Bb) \ar@{^{(}->}[d] \ar@{->>}[r] & \cO(\Gb_a^3) \ar@{^{(}->}[d]
\\
 & \cO(\mathfrak{G})  \ar@{^{(}->}[r] \ar@{->>}[d] & \widetilde{\cD} \ar@{->>}[r] \ar@{->>}[d] & U(\mathfrak{osp}(1|2)) \ar@{->>}[d]
\\
& \nucleo \ar@{^{(}->}[r] & \cD \ar@{->>}[r] & \ugo(\mathfrak{osp}(1|2))
}
\end{aligned}
\end{align}
\begin{prop}\label{thm:diagram-exact sequences}
All columns and rows in \eqref{eq:diagram-exact sequences} are exact sequences. 
\end{prop}

\begin{proof}
Theorems \ref{prop:Db-exact} and \ref{prop:cD-as-super-abelian-extension}, and
and Propositions \ref{prop:cD-as-Hopf-superalgebra} and \ref{prop:cOG-exact} covers everything except the topmost row and the rightmost column.
 
For the rightmost column we need to prove that the even Hopf subalgebra  $Z' = \left\langle e^p, f^p, h^p-h \right\rangle$ of $U(\mathfrak{osp}(1|2))$ is $\cO(\Gb_a^3)\simeq \ku[X_1,X_2,X_3]$. Taking the
basis of $U(\mathfrak{osp}(1|2))$ consisting of monomials
\begin{align*}
f^n (h^p-h)^k h^\ell e^m \psi_+^i \psi_-^j
\end{align*} 
with $(n,m,k,\ell,i,j)\in \N_0^3\times\I_{0, p-1}\times\I_{0, 1}^2$, we 
see that the assignment
\begin{align*}
X_1&\mapsto f^p, & X_2&\mapsto h^p - h, & X_3&\mapsto e^p, 
\end{align*}
gives an algebra isomorphism $Z'\simeq \ku[X_1,X_2,X_3] \simeq \cO(\Gb_a^3)$. Comparing comultiplications, the previous isomorphism is of Hopf algebras. $\cO(\Gb_a^3)$ is stable 
by the adjoint action of  $U(\mathfrak{osp}(1|2))$ and a free module over $\cO(\Gb_a^3)$ using the previous basis, then Remark \ref{remark-exact-sequence-hopf} applies and the column is exact.

We next describe
explicitly the top row. $\phi\colon \cO(\Gb) \to \cO(\Bb)$ is given by
\begin{align*}
\mathrm{X_1}&\mapsto X_1, &\mathrm{X_2}&\mapsto X_4, &\mathrm{T}&\mapsto T.
\end{align*} 
Recall that $\cO(\Bb)\simeq \ku[T^{\pm},X_1,\dots,X_5]$ and $\cO(\Gb)\simeq \ku[\mathrm{X_1},\mathrm{X_2}, \mathrm{T}^{\pm 1}]$, cf. Theorem \ref{thm:Z} and Proposition \ref{prop:cOG-exact}. Take
$\psi \colon \cO(\Bb) \longrightarrow \cO(\Gb_a^3)$ given by 
\begin{align*}
T&\mapsto 1,& X_1 &\mapsto 0,& X_2 &\mapsto -f^p,&
X_3&\mapsto h^p-h,& X_4 &\mapsto 0,& X_5 &\mapsto e^p.
\end{align*}
Now $\phi$ is a injective because it maps a basis into a linearly independent set 
while $\psi$ is surjective since the PBW-basis of $\cO(\Gb_a^3)$ is in its image. 
Clearly $\ker \psi \supseteq \cO(\Bb) \phi(\cO(\Gb))^+$, the other inclusion follows using the basis of $\cO(\Bb)$ 
\begin{align*}
 \tau_m\,X_1^{n}\,X_2^{r}\,X_3^k X_4^{i} X_5^{j}, &&  i,j,k,n,r\in\N_0, m\in \Z,
\end{align*}
where $\tau_m$ is given by the formula
$\tau_m = \begin{cases}
(T-1)^m & \text{if } m \geq 0,\\
(T^{-1}-1)^{-m} & \text{if } m < 0
\end{cases}$.
The adjoint action of $\cO(\Bb)$ is trivial, so $\phi(\cO(\Gb))$ is invariant. Considering the PBW-bases 
of $\cO(\Bb)$ and $\cO(\Gb)$, $\phi$ is faithfully flat. 
Thus the row is exact.
\end{proof}

\end{document}